\documentclass[11pt,reqno]{amsart}
\usepackage{a4wide}
\usepackage{amsmath, amsthm, xcolor}
\usepackage{amssymb}
\usepackage{graphicx,soul}
\usepackage[normalem]{ulem}
\usepackage{relsize}
\usepackage{enumitem}
\newcommand{\hidden}[1]{}
\usepackage{verbatim}

\usepackage{color}
\usepackage[mathscr]{euscript}
\usepackage{tikz}

\newcommand{\R}{{\Bbb R}}

\newcommand{\N}{{\Bbb N}}
\newcommand{\C}{{\Bbb C}}

\newcommand{\D}{{\Bbb D}}

\newtheorem{question}{Question}
\newtheorem{theorem}{Theorem}[section]

\newtheorem{lemma}[theorem]{Lemma}
\newtheorem{proposition}[theorem]{Proposition}
\newtheorem{definition}[theorem]{Definition}

\newtheorem{TAA}{Theorem}

\theoremstyle{remark}
\newtheorem*{remark}{Remark}
\newtheorem*{example}{Example}

\numberwithin{equation}{section}

\newcommand{\cR}{\mathcal{R}}

\begin{document}

 \author{Ayesha Bennett}

\bigskip

\date{}

\begin{abstract}
We investigate the shrinking target and recurrence set associated to non-autonomous measure-preserving systems on compact metric spaces, establishing zero-one criteria in the spirit of classical Borel-Cantelli results.

Our first main theorem gives a quantitative shrinking target result for non-autonomous systems under a uniform mixing condition, providing asymptotics with an optimal error term. This general result is applicable to certain families of inner functions, yielding concrete applications such as patterns of zeros in the multibase expansion.
Turning to recurrence, we establish new zero-measure laws for non-autonomous systems. In the autonomous case, we prove a zero-one criterion for recurrence sets of centred, one-component inner functions via Markov partitions and distortion estimates. Together, these results provide a unified framework for shrinking target and recurrence problems in both autonomous and non-autonomous dynamics.
\end{abstract}

\title{ The shrinking target and recurrence problem for  non-autonomous systems }
 \maketitle
\section{Introduction}
Let $(X, d)$ be a compact metric space equipped with a probability measure $\mu$ and associated Borel $\sigma$-algebra $\mathcal{B}$. 
Given a sequence of measure-preserving transformations $T_\N:=\{t_n\}_{n \in \N}$, we define
$ T_n := t_n \circ t_{n-1} \circ \dots \circ t_1
$ with $n \in \N$ and we write $(X,\mathcal{B},\mu,T_\N,d)$ for the associated measure-preserving, non-autonomous system. In the case when $t_n=t$ for all $n \in \N$,  we regain  the `standard' notion of a measure-preserving dynamical system, $(X,\mathcal{B},\mu,t,d)$. Throughout, we will only consider measure-preserving systems, and will therefore omit this descriptor.
Let $\{r_n\}_{n \in \N}$ be a positive, real sequence. We are interested in the following two sets:
\begin{definition}
Given a sequence of balls $\{B(x_n,r_n)\}_{n \in \N}$ contained in $X$ with centres $x_n \in X$, 
we let
\begin{eqnarray}\label{defshrink}
  \mathcal{W}(T_\N,\{x_n, r_n\}) :=  \left\{  x \in X :  T_{n}(x) \in B(x_n,r_n) \hbox{ for infinitely many }n\in \mathbb{N} \right\} \, ,
\end{eqnarray}
and refer to $\mathcal{W}(T_\N,\{x_n, r_n\})$ as the shrinking target set associated to $T_\N$.
\end{definition} 
\begin{definition}
    We let
\begin{eqnarray}\label{defrecc}
  \cR( T_\N,\{r_n\} ) := \left\{  x \in X :  T_{n}(x) \in B(x,r_n) \hbox{ for infinitely many }n\in \mathbb{N} \right\} \, , 
\end{eqnarray}
and refer to $\cR( T_\N,\{r_n\} )$ as the recurrence set associated to $T_\N$.
\end{definition}

Consider the most basic case, when we set $T_\N=\{t\}$, so we have an \emph{autonomous} system, and $r_n=c$ for all $n$  in \eqref{defshrink} and \eqref{defrecc}. By Poincar\'{e} Recurrence, the recurrence set has full $\mu$-measure. If $t$ is ergodic, then by an application of the Ergodic Theorem, the shrinking target set has full $\mu$-measure. 
 In view of this, it is natural to ask what happens to the $\mu$-measure of the sets $\mathcal{W}(t,\{x_n, r_n\})$ and $\cR(t,\{r_n\})$ if, instead of being constant, $r_n \to 0$ as $n \to \infty$. Investigating this question has led to a rich theory in the autonomous setting.   In particular, under the additional assumption  that $\mu$ is $\delta$-Ahlfors regular ( defined by \eqref{defahl}), one seeks zero-one criteria for $\mathcal{W}(t,\{x_n, r_n\})$ and $\mathcal{R}(t,\{r_n\})$  based on the convergence of $\sum_{n=1}^\infty r_n^\delta$. The sets associated with \eqref{defshrink} and \eqref{defrecc} can be expressed as $\limsup$ sets, and the sought after criteria are in line with standard Borel-Cantelli statements from probability theory -- see Section~\ref{bc} for more discussion.

The concept of shrinking targets was first introduced in \cite{hv} by Hill-Velani as a far-reaching `dynamical' generalisation of the classical framework of metric Diophantine approximation. They established bounds on the Hausdorff dimension of the Julia set of an expanding rational map, and thereby laying the foundation for the study of shrinking target sets within the context of the Borel-Cantelli Lemmas. 
Subsequent developments have focused on two main directions. The first is to determine under what conditions a measure-preserving system satisfies a zero–one criterion (see \cite{BERFS, FMP, kim2007, LLVZ} and references within). The second, arising when the shrinking target set has measure zero, is to establish bounds on its Hausdorff dimension (see \cite{baker2024, li2014} and references within). In this paper we concentrate on the first problem. Of particular note is \cite{FMP} in which Fern\'{a}ndez-Meli\'{a}n-Pestana prove an asymptotic statement for measure of the shrinking target set $\mathcal{W}(t, \{x_0, r_n\})$ associated to an autonomous system satisfying a uniform mixing condition. From this, they derive a zero-one criterion based on the convergence of $\sum_{n=1}^\infty \mu(B(x_0,r_n))$. This was strengthened to a quantitative result with an essentially optimal error term in \cite{LLVZ} by exploiting a quantitative form of the Borel-Cantelli Lemma (see  Section~\ref{bc}, Lemma~\ref{ebc}).

In contrast, study of the \emph{non-autonomous} shrinking target problem is  far less complete. Even when $r_n = c$, and $t_n$ ergodic for all $n$, full measure of $\mathcal{W}(T_\N,\{x_n, r_n\})$ and $\mathcal{R}(T_\N,\{r_n\})$ need not hold, since compositions of measure-preserving maps can degrade ergodic properties--see Section~\ref{autnonaut} for discussion. Nevertheless, recent work has begun to address such problems. 
In \cite{simurb}, Fishman-Mance-Simmons-Urba\'{n}ski study the Cantor series expansion of points in $[0,1]$, which we define later before  stating Theorem~\ref{ABcf}. They give bounds on the Hausdorff dimension of the associated non-autonomous shrinking target set. Subsequently, in the same context, Sun \cite[Theorem~1.1]{Sun}  provides a zero-one criteria for the set.

 In complex dynamics, attention has turned to shrinking target problems for non-autonomous systems of inner functions. These functions are holomorphic self-maps of the unit disk that extend Lebesgue almost everywhere to maps of the boundary (see Section~\ref{preres}, Definition~\ref{definner}), and are closed under composition \cite[Theorem~2]{ryff}.    
They provide a natural framework for bounded holomorphic dynamics as every bounded holomorphic map admits a canonical outer–inner function factorisation.
A further motivation of their study comes from investigating entire maps with simply connected wandering domains of entire functions. Here, under iteration, the domain `wanders' through $\hat\C$, and repeated application of the Riemann Mapping Theorem transfers the dynamics to the closed unit disk, inducing a non-autonomous system of inner functions.
This perspective underlies works such as \cite{BERFS}, in which the authors establish a zero-one criterion as well as a variety of zero and full measure laws in this non-autonomous inner function setting. Here, we use the term zero (or full) measure law to refer to a result that provides sufficient conditions under which the set of interest has zero (or full) measure.

Motivated by these developments, we prove a general quantitative shrinking target zero–one criterion, with explicit error terms for non-autonomous dynamical systems satisfying the following `uniform mixing' condition.

\begin{definition}\label{ABunimix}
 Let  $(X,\mathcal{B},\mu,T_{ \N},d)$ be a dynamical system and let $\{x_n\}_{n \in \N}$ be a sequence of points in $X$. We say that $T_\N$ is uniformly mixing at $\{x_n\}$ if there exists a  real, positive function $\phi: \N \to \R^+$ with $ \sum_{n=1}^{\infty} \phi(n) < \infty$ such that for each $n \in \N$
 \\
 \begin{equation}\label{unimixing1}
\left|\frac{\mu(B\cap (t_n \circ \dots \circ t_m)^{-1}(A))}{\mu(A)}-\mu(B)\right| \,  \le  \,  \phi(n-m-1) \, 
\end{equation}
\\
 for all $1 \le m<n$, all balls $B \subset X$ with centre at $x_n$ and all $A \in \mathcal{B}$ with $\mu(A)>0$. We say that $T_\N$ is uniformly mixing if it is uniformly mixing at any given sequence $\{x_n\}$ in $X$.
\end{definition}
\begin{remark}\label{fmpmix}
Setting $t_n=t$ and $x_n=x_0$ for all $n$, we recover the notion of uniform mixing at a point introduced in \cite[Definition~2]{FMP}. If we set $\phi(n)=\gamma^n $, for $0<\gamma<1$, then $T_\N$ is exponentially mixing in the usual sense. 
\end{remark}

 With Definition~\ref{ABunimix} in hand, we state our quantitative shrinking target theorem for non-autonomous systems. 

\begin{TAA}\label{ABthA}
Let $(X,\mathcal{B},\mu,T_\N,d)$ be a dynamical system. Let $\{B(x_n,r_n)\}_{n \in \N}$ be a sequence of balls in $X$ with centres $x_n$ and decreasing radii $r_n$. Suppose that $T_\N$ is uniformly mixing at $\{x_n\}_{n \in \N}$.
Then, for any given $\varepsilon>0$, we have that for $\mu$-almost every $x \in X$
\begin{eqnarray}\label{asymt}
   ~\hspace{4.5mm} \# \big\{ 1\le n \le   N :    T_{n}(x) \in B(x_n,r_n) \big\} =\Phi(N)+O\Bigl(\Phi^{1/2}(N) \ (\log\Phi(N))^{3/2+\varepsilon}\Bigr) \, ,
\end{eqnarray}
 where $\Phi(N):=\sum\limits_{n=1}^N \mu (B(x_n,r_n)) $.
In particular,
 \begin{eqnarray*}
   \mu\big( \mathcal{W}(T_\N,\{x_n, r_n\}) \big) = \begin{cases}
        0 &\text{ \ if}\ \  \sum\limits_{n=1}^\infty \mu(B(x_n,r_n)) <\infty \, , \\[2ex]
        1 &\text{ \ if}\ \  \sum\limits_{n=1}^\infty\mu(B(x_n,r_n))
        =\infty \, .
    \end{cases}
	\end{eqnarray*}
 
\end{TAA}
Recall that $\mu$ is $\delta$-Ahlfors regular if there exist absolute constants $c_1,c_2>0$ such that for any ball $B(x,r)$ with $x \in X$, \begin{eqnarray}\label{defahl}
    c_1r^\delta \leq \mu(B(x,r)) \leq c_2r^\delta \, .
\end{eqnarray}
Equivalently, we write $\mu(B(x,r)) \asymp r^\delta$.
Note that if $\mu$ is $\delta$-Ahlfors regular then the above criterion is determined by the summability of $\sum_{n=1}^\infty r_n^\delta$.

 Theorem~\ref{ABthA} is a quantitative version of \cite[Theorem~1]{FMP}, and a generalisation of \cite[Theorem~1]{LLVZ} to the non-autonomous setting.
As a concrete application, we specialise Theorem~\ref{ABthA} to a system  of inner functions. Throughout, $\D$ will denote the open unit disk in $\C$, $\partial \D$ its boundary (i.e. the unit circle) and $\lambda$ is normalised Lebesgue measure. Consider a sequence of centred, uniformly contracting inner functions $F_\N := \{f_n\}_{n \in \N}$, meaning $f_n(0) = 0$ and $|f_n'(0)| \le \alpha < 1$ for all $n \in \N$.  
Then, we will show in Section~\ref{preres} that $(\partial\D, \mathcal{B}, \lambda, F_\N, d)$ is an exponentially mixing system, where $d$ is normalised arc length.
\begin{TAA}\label{ABthB}
Let $F_\N $ be a sequence of centred, uniformly contracting inner functions. Let  $\{B(z_n,r_n)\}_{n \in \N}$ be a sequence of balls in $\partial \D$ with centres $z_n$ and decreasing radii $r_n$.
Then, for any given $\varepsilon>0$, we have
\begin{equation*} \# \big\{ 1\le n \le   N :  F_n(z) \in B(z_n,r_n)\big\}  =\Phi(N)+O\left(\Phi^{1/2}(N) \ (\log\Phi(N))^{3/2+\varepsilon}\right)
\end{equation*}
\noindent for $\lambda$-almost every $z\in \partial\D$, where
$\Phi(N):=\sum\limits_{n=1}^N r_n$.
In particular,
 \begin{eqnarray*}
   \lambda \big( \mathcal{W}(F_\N,\{z_n, r_n\}) \big) = \begin{cases}
        0 &\text{ \ if}\ \  \sum\limits_{n=1}^\infty r_n <\infty \, ,\\[2ex]
        1 &\text{ \ if}\ \  \sum\limits_{n=1}^\infty r_n
        =\infty \, .
    \end{cases}
	\end{eqnarray*}
 
\end{TAA}
The ``In particular'' part of Theorem~\ref{ABthB} coincides with a recent result  of  Benini et al [3, Theorem C]. It is worth highlighting, that although not stated, their  proof  actually gives the stronger counting statement.    

As an application of Theorem~\ref{ABthB}, we prove a zero-one criterion on the occurrence of arbitrarily long string of zeros in the multibase or `Cantor series' expansion of points in the unit interval.
This expansion is the unique expansion of $x \in [0,1)$ with respect to the sequence of bases $\{b_i\}_{i \in \N}$ with $b_i \in \N_{\geq 2}$ such that
$$
x = \sum_{i=1}^{\infty} \frac{d_i}{b_1\dots b_i} \, , \  d_i \in \{0,\dots,b_i-1\} 
$$
where the digits $d_i \neq b_i -1$ for all $i \geq i_0\in \N$.
We will denote this expansion by $x= [d_1,..,d_n,\dots]$ and with this notation in place, we can state our theorem. 

\begin{TAA}\label{ABcf}
Let $\{b_i\}_{i \in \N}$ be a sequence of bases with $b_i \in \N_{\geq 2}$, and $\xi: \N \to \N$ be an increasing function. Then, for any given $\varepsilon>0$, we have that for almost every $x=[d_1,\dots,d_n,\dots] \in [0,1)$,
$$ \# \big\{ 1\le n \le   N :   d_{n+1}=\dots=d_{n+\xi(n)}=0 \big\} =\Phi(N)+O\left(\Phi^{1/2}(N) \ (\log\Phi(N))^{3/2+\varepsilon}\right) \, ,
$$  where $\Phi(N):=\sum\limits_{n=1}^N \frac{1}{b_{n+1}\dots b_{n+1+\xi(n)}}$. 
In particular, 
\begin{eqnarray*}
   \lambda \big(D(\xi)\big) = \begin{cases}
        0 &\text{ \ if}\ \  \sum\limits_{n=1}^\infty \frac{1}{b_n\dots b_{n+1+\xi(n)}} <\infty \, ,\\[2ex]
        1 &\text{ \ if}\ \ \sum\limits_{n=1}^\infty \frac{1}{b_n\dots b_{n+1+\xi(n)}}
        =\infty \, ,
    \end{cases}
	\end{eqnarray*}
    where $D(\xi)$ is $x=[d_1,\dots,d_n]\in [0,1)$ for which $d_{n+1}=d_{n+2}=\dots=d_{n+\xi(n)}=0$ 
        for infinitely many $n \in \N $.
\end{TAA}
In \cite{mance}, Mance provides an asymptotic result pertaining to the typicality of normal numbers in the multibase expansion for sequence of bases tending to infinity. More precisely, he quantifies the frequency of blocks (not necessarily zeros) of fixed length in such expansions. Therefore, the more `interesting’ case of Theorem~\ref{ABcf} is when the function $\xi$ is unbounded, since, in this case, the  result concerns the occurrence of arbitrarily long sequences of zeros in the multibase expansion. We discuss this further following the proof of Theorem~\ref{ABcf} in Section~2.6. Note that on specialising to $b_i=b$ for all $i$, we return to  the autonomous case, where our result concerns the base
$b$ expansion of points  $x\in[0,1)$. In this case, we can readily strengthen the conclusion of \cite[Theorem~5]{FMP},  a related result concerning the frequencies of arbitrarily long strings of digits (not necessarily zeros) in the base $b$ expansion, by upgrading it to an asymptotic statement with optimal error.
 \medskip
 
We now turn our attention to the recurrence set $\cR(T_\N,\{r_n\})$. Recall that for (measure-preserving) autonomous systems $(X,\mathcal{B},\mu,t,d)$, Poincar\'e Recurrence guarantees $\mu(\cR(t,\{r_n\}))=1$ when $\{r_n\}_{n \in \N}$ is a constant sequence. In the case $r_n \to 0$, in pioneering work, Boshernitzan \cite{boshernitzan1993} established full measure results for $\cR(t,\{r_n\})$ under mild additional assumptions on $X$, beyond mere measure-preservation.
Complete zero–one criteria, however, have only been obtained for autonomous dynamical systems that, at minimum, are uniformly mixing (see Definition~\ref{ABunimix}) and admit a `regular' partition (see Section~\ref{Markov}). Such systems  include those considered within \cite{baker2021}, \cite{chang2019}, \cite{He}, \cite{kleinbock2023} and, particularly relevant to this paper, \cite{HeLi} in which He-Liao considered piecewise expanding maps with a finite partition, and \cite{HLSW}, in which Hussain-Li-Simmons-Wang consider maps with a countable partition and regularity properties.
 The study of non-autonomous recurrence, however, remains virtually unexplored.
 We remedy this gap with the following result, which provides a zero measure criteria for the recurrence set of a non-autonomous system.
\begin{TAA}\label{ABconv}
Let $(X,\mathcal{B},\mu,T_{\N}, d)$ be a compositional system, and suppose that $\mu$ is a $\delta$-Ahlfors regular probability measure, and $T_\N$ is uniformly mixing with respect to $\mu$.  Let $\{r_n\}_{n\in \N} $ be a sequence of positive, real numbers. Then 
 if  $\sum_{n=1}^\infty r_n^{\delta} <\infty $,  then 
$$
		\mu(\mathcal{R} (T_\N,\{r_n\} ))=
			0 \, .
$$
\end{TAA}
An immediate corollary of Theorem~\ref{ABconv} is the following  zero measure law for non-autonomous systems of inner functions. In this setting, to the best of our knowledge, there has not been any work on the recurrence set outside of when $\{r_n\}$ is a constant sequence, even for autonomous systems. 
\begin{TAA}\label{ABconvin}
Let $(\partial \D, \mathcal{B},\lambda, F_{\N},d)$ be a dynamical system with $F_\N$ a sequence of centred, uniformly contracting inner functions. Let $\{r_n\}_{n\in \N} $ be a sequence of positive, real numbers such that $\sum_{n=1}^\infty r_n <\infty $.
     Then 
$$ 
\lambda(\cR(F_\N,\{r_n\})) = 0 \, .$$
\end{TAA}
It is natural to ask if either Theorem~\ref{ABconv} and Theorem~\ref{ABconvin} can be upgraded to a zero-one criterion for their respective non-autonomous systems. A result of Shen \cite[Theorem~1.2] {nonautrec} gives a zero-one criterion for the recurrence set associated to the Cantor series expansion of points in the unit interval gives a partial upgrade of Theorem~\ref{ABconvin} for maps of the form $z^{b_n}$. Proving a complete criterion likely requires defining a sufficiently regular partition associated to the sequence of (potentially infinite degree) maps that generate the non-autonomous systems in question.
\\
The following theorem provides a complete zero-one criterion for autonomous systems associated to a special class of functions called centred, one-component inner functions (see Definition~\ref{definner} in Section~\ref{preres}). We remark that this class of inner functions contains well known classes of functions, such as all centred, finite Blaschke products -- see Lemma~\ref{foc}.

\begin{TAA}\label{ABrecocif}
   Consider the dynamical system $(\partial \D, \mathcal{B},\lambda,f,d)$, where $f$ is a one-component, centred inner function. Let $\{r_n\}_{n \in \N}$ be a sequence of positive, real numbers. Then 
   \begin{eqnarray*}
   \lambda(\cR(f,\{r_n\}) )= \begin{cases}
        0 &\text{ \ if}\ \  \sum\limits_{n=1}^\infty r_n <\infty \,  ,\\[3ex]
        1 &\text{ \ if}\ \  \sum\limits_{n=1}^\infty r_n
        =\infty \, .
    \end{cases}
	\end{eqnarray*}
\end{TAA}
This proof is split into two cases, depending on if $f$ is of finite or infinite degree. The proof in the finite degree case follows an application of the previously mentioned \cite[Theorem~1.3]{HeLi}. 
The proof for infinite degree one-component inner functions relies on the fact that they admit a countable Markov partition as shown in \cite{urb}. We recall this construction in Section~\ref{Markov}, and we then prove that the partition satisfies various distortion estimates necessary to prove the theorem. The overall strategy follows that of \cite[Theorem~1.3]{HLSW}, but since the extension of $f$ to the unit circle fails to be a self-map on a set of Lebesgue measure zero, the result does not follow directly, and requires non-trivial work.

\medskip
The structure of the paper is as follows. Section~\ref{shrink} contains our shrinking target results. We begin in Section~\ref{bc} with background on the Borel–Cantelli Lemmas, needed to prove all of the aforementioned theorems. We start by proving Theorem~\ref{ABthA} in Section~\ref{pfthm4}. We then introduce preliminaries on inner functions in Section~\ref{preres}, which we apply in Section~\ref{proofb} to prove Theorem~\ref{ABthB}. In Section~\ref{backbase} we review base expansions, leading to the proof of Theorem~\ref{ABcf} in Section~\ref{apply}.
Section~\ref{nonrecres} establishes our non-autonomous recurrence results (Theorems~\ref{ABconv} and \ref{ABconvin}), while Section~\ref{autrecres} treats the autonomous zero–one criterion. There we first recall background on one-component inner functions (Section~\ref{backocif]}) before constructing their Markov partitions (Section~\ref{Markov}), and finally prove Theorem~\ref{ABrecocif} in Section~\ref{proofrecocif}. The paper concludes in Section~\ref{autnonaut} with a discussion of differences between shrinking target and recurrence results in the autonomous and non-autonomous settings.

\medskip
\section{Shrinking target background and proofs}\label{shrink}
We begin with a simple but important observation. We can express both the shrinking target set in \eqref{defshrink} and recurrence set in \eqref{defrecc} as $\limsup$ sets: 
$$\mathcal{W}(T_\N,\{x_n, r_n\}) = \limsup\limits_{n \to \infty}E_n \quad \text{ \ and \ } \quad 
\cR(T_\N,\{r_n\})= \limsup\limits_{n \to \infty} A_n \, $$  where for each $n \in \N$ $$ ~\hspace{4.5mm} E_n:= \{ x \in X: T_n(x) \in B(x_n,r_n)\} \quad \text{ \ and \ } \quad  A_n :=  \{x \in X: T_n(x) \in B(x,r_n)\} \, .$$ 
We will refer to $E_n$ and $A_n$ as the `building block sets' of the respective $\limsup$ sets. 
We will later show that proving our zero–one criteria and quantitative estimates reduces to establishing Borel–Cantelli type bounds for the corresponding building block sets. We state the relevant probabilistic results below.
\subsection{Borel-Cantelli Lemmas\label{bc}}
As before, $(X, \mathcal{B},\mu)$ will be a probability space, and now we let $P_n$ denote a sequence of events in $X$.
We start by stating the First Borel-Cantelli Lemma. 
\begin{lemma}[First Borel-Cantelli Lemma]\label{cbc}
 Suppose that $\sum_{n=1}^\infty \mu(P_n)$ converges. Then $\mu(\limsup\limits_{n \to \infty} P_n)=0$.
\end{lemma}
 The following example shows that $\sum_{n=1}^\infty \mu(P_n)$ can diverge, and yet the measure of the associated $\limsup$ set can be zero. 
\begin{example}
    Consider the space $[0,1]$ with Lebesgue measure $\lambda$ and $P_n=[1-\frac{1}{n},1)$. Then $\sum_{n=1}^\infty \lambda(P_n) = \sum_{n=1}^\infty \frac{1}{n} = \infty$, however $\limsup\limits_{n \to \infty} P_n =\emptyset \, .$
\end{example}
Therefore, for a positive measure result, we require the sets $P_n$ to be in some sense independent.
    The below lemma, often referred to as the Divergence or Second Borel-Cantelli Lemma quantifies the necessary independence condition.
    \begin{lemma}[Divergence Borel-Cantelli Lemma]\label{dbc}
         Suppose $\sum_{n=1}^\infty\mu(P_n)=\infty$ and there exists a positive constant $ C\geq1$ such that for infinitely many $N \in \N$
    \begin{eqnarray}\label{qioa}
        \sum_{1 \leq s,t \leq N}\mu\big(P_s \cap P_t \big) \: \leq \:  C\Bigl(\sum_{s=1}^N\mu(P_s) \Bigr)^2\, .
    \end{eqnarray}
 Then $\mu\big(\limsup\limits_{n \to \infty} P_n\big) \geq C^{-1}$.
    \end{lemma}
The proof hinges on the Cauchy-Schwarz Inequality, see \cite[Lemma~3]{Harman2} for details. We refer to \eqref{qioa} as a `quasi-independence on average' condition.
Lemma~\ref{dbc} only allows us to conclude positive, not full measure, unless $C=1$. However, often we can jump from positive measure to a full measure result via theorems such as Kolmorogorov's Zero-One Law.

The following lemma, which can be found in \cite[pg 13]{harman}, is sometimes referred to as a quantitative Borel-Cantelli Lemma.
\begin{lemma} \label{ebc}
Let $\{h_n(x)\}_{n \in \N}$ be a sequence of non-negative $\mu$-measurable functions defined on $X$, and $\{h_n\}_{n \in \N },\ \{\phi_n\}_{n  \in \N}$ be sequences of real numbers  such that
\begin{equation*}  0\leq f_n \leq \phi_n \hspace{7mm} (n=1,2,\ldots).  \end{equation*}
\noindent
Suppose that for arbitrary  $a,b \in \N$ with $a <  b$, we have
\begin{equation} \label{ebc_condition1}
\int_{X} \Bigl(\sum_{a \leq n \leq b} \big( h_n(x) -  h_n \big) \Bigr)^2\mathrm{d}\mu(x)\, \leq\,  C\!\sum_{a \leq n\leq b}\phi_n
\end{equation}

\noindent for an absolute constant $C>0$. Then, for any given $\varepsilon>0$,  we have, as $N \to \infty$,
\begin{equation} \label{ebc_conclusion}
\sum_{n=1}^N h_n(x)\, =\, \sum_{n=1}^{N}h_n\, +\, O\Bigl(\Phi(N)^{1/2}\log^{\frac{3}{2}+\varepsilon}\Phi(N)+\max_{1\leq n\leq N}h_n\Bigr)
\end{equation}
\noindent for $\mu$-almost every $x\in X$, where $
\Phi(N):= \sum\limits_{n=1}^{N}\phi_n
$. 

\end{lemma}
Note that in statistical terms, if $h_n$ is the mean of $h_n(x)$, then \eqref{ebc_condition1} corresponds to bounding the variance.

\subsection{Proof of Theorem~\ref{ABthA}} \label{pfthm4}
The proof of Theorem~\ref{ABthA} relies on Lemma~\ref{ebc}. We now observe that  $E_n=T_n^{-1}(B(x_n,r_n))$, and so by the measure-preserving property of $T_n$, 
\begin{equation}\label{condmpt}
\mu(E_n) = \mu(B(x_n,r_n)) \qquad \forall \; n \, \in \; \N \, .
\end{equation}
This makes estimating the measure of the building block sets straightforward.

\begin{proof}[Proof of Theorem~\ref{ABthA}]
To prove the asymptotic statement of the theorem, we apply  Lemma~\ref{ebc} to the functions
\begin{eqnarray*}
h_n(x) := \chi_{E_n}(x)  \quad {\rm and } \quad 
h_n =\phi_n := \mu(E_n) 
\end{eqnarray*}
where $\chi$ is the characteristic function.
With this choice of functions, the LHS of \eqref{ebc_conclusion}  $$\#\{1\leq n \leq N: T_n(x) \in B(x_n,r_n)  \} \, ,$$ which is exactly the LHS of \eqref{asymt}.

We will now show that \eqref{ebc_condition1} is satisfied.
First, note that on expanding \eqref{ebc_condition1} with our choice of functions, the condition becomes
\begin{equation}\label{sumfun}
    \sum_{a \leq n \leq b} \mu(E_n) + 2\sum_{a \leq m < n \leq b} \mu(E_m \cap E_n)- \Bigl(\sum_{a \leq n \leq b}\mu(E_n)\Bigr )^2 \leq C\sum_{n=1}^\infty\mu(A_n) \, .
\end{equation}
It is clear then that if we can show that 
\begin{eqnarray}\label{sumfun1}
    2\sum_{a \leq m < n \leq b} \mu(E_m \cap E_n) \leq \Bigl(\sum_{a \leq n \leq b}\mu(E_n)\Bigr )^2  + C\sum_{a \leq n \leq b}\mu(E_n)
\end{eqnarray}
for some $C>0$, then \eqref{sumfun} is satisfied. With this in mind, we start by estimating the intersection term $\mu(E_m \cap E_n)$. Writing $n= m +l$ for some $l \in \N >0$, we have $T_{n}^{-1} = T_{m+l}^{-1} = T_m^{-1} \circ t_{m+1}^{-1} \circ \dots t_{m+l}^{-1} \, $. This, together with \eqref{condmpt}, gives
\begin{eqnarray*} \label{intbdb}
\mu(E_m \cap E_n) &=&
\mu(T_{m}^{-1}(B_m) \cap T_{m+l}^{-1}(B_{m+l})) \nonumber\\[2ex] &=& \mu\left(T_{m}^{-1}(B_m \cap t_{m+1}^{-1} \circ \cdots \circ t_{m+l}^{-1}(B_{m+l}))\right)  \nonumber \\[2ex] &=& \mu(B_m \cap t_{m+1}^{-1} \circ \cdots \circ t_{m+l}^{-1}(B_{m+l})) \, .
\end{eqnarray*}
We are now in a position to apply the fact that $T_\N$ is uniformly mixing, and so
\begin{equation} \label{sumbdb}
\begin{split}\mu(E_m \cap E_n) &\leq \mu(B_m) \mu(B_{m+l}) + \mu(B_m)\phi(l) \\[2ex]&=  \mu(B_m) \mu(B_{n}) + \mu(B_m)\phi(n-m) \, .
\end{split}
\end{equation}
Therefore, we can apply \eqref{condmpt} to \eqref{sumbdb}, and then substitute this in \eqref{sumfun} to get
\begin{eqnarray*}
    2\sum_{a \leq m < n \leq b}\mu(E_m \cap E_n)   &\leq & 
    2\sum_{a \leq m < n \leq b} \mu(E_m) \mu(E_n) + 2c \sum_{a \leq n \leq b} \mu(E_n) \\[2ex] & \leq & \Bigl(\sum_{a \leq n \leq b}\mu(E_n)\Bigr )^2 + 2c \sum_{a \leq n \leq b} \mu(E_n) 
\end{eqnarray*}
where $c = \sum\limits_{n=1}^\infty \phi(n)$. Therefore, we have shown \eqref{sumfun1} holds, and so \eqref{sumfun} is satisfied with $C= 2c+1$.
 Hence, on applying Lemma~\ref{ebc}, we have for $\mu$-almost every $x \in X$
\begin{eqnarray*}
\#\{1\leq n \leq N: T_n(x) \in B(x_n,r_n)  \}\, &=&\, \sum_{n=1}^{N} \mu(E_n) \, \\[2ex] & ~ \hspace*{-12ex} +\,&    ~ \hspace*{-5ex} O\biggl(\Bigl(\sum_{n=1}^{N} \mu(E_n)\Bigr)^{\frac{1}{2}}\log^{\frac{3}{2}+\varepsilon}\sum_{n=1}^{N} \mu(E_n) + \max_{1 \leq n \leq N}{\mu(E_n)} \biggr ) \, .
\end{eqnarray*}

\noindent Clearly, $\mu(E_n) \leq 1$ for all $n \in \N$, and so on using  \eqref{condmpt} once more, we obtain the desired asymptotic statement of Theorem~\ref{ABthA}.

We finish by establishing the `In particular' part, namely the zero-one criteria. We remark that with observation \eqref{condmpt}, we can obtain the zero measure statement immediately from Lemma~\ref{cbc}. However, the complete statement also follows directly from the asymptotic estimate \eqref{asymt} as we now show. 
Given a point $x \in X$, if
$$\lim_{N \to \infty}\#\{1\leq n \leq N: T_n(x) \in B(x_n,r_n)  \}\, = \infty \, ,$$ 
then  $x$ enters $E_n$ infinitely often, and so $x \in \mathcal{W}(T_\N,\{x_n, r_n\})$. 
Indeed, the LHS of \eqref{asymt} is finite (resp. infinite) as $N \to \infty$ precisely when $\lim_{N \to \infty}\sum_{n=1}^N \mu(B(x_n,r_n))$ converges (resp. diverges). Therefore, for $\mu$-almost every $x \in X$ we have
\begin{eqnarray*}
x \in \mathcal{W}(T_\N,\{x_n, r_n\}) & \iff & \lim_{N \to \infty}\#\{1\leq n \leq N: T_n(x) \in B(x_n,r_n)  \} = \infty  \\ &\iff& \lim_{N \to \infty}\sum_{n=1}^N \mu(B(x_n,r_n)) = \infty\, 
\end{eqnarray*}
and so the zero-one criterion follows.
\end{proof}
\medskip

\subsection{Background on inner functions}\label{preres}
In this section, we provide the background material required to prove our shrinking target criterion for inner functions. At the same time, this will allow us to establish the necessary notation.
Before we can prove Theorem~\ref{ABthB} (which follows as an application of Theorem~\ref{ABthA}), we must first identify which inner functions induce measure-preserving  systems on $\partial \D$ with suitable mixing properties. Readers already familiar with the basics of inner functions may wish to skip straight to the proof of Theorem~\ref{ABthB} in Section~\ref{proofb}.

\begin{definition}\label{definner}
    An inner function is a holomorphic  map of $f: \D \to \D$  such that for $\lambda$ almost every $e^{i \theta} \in \partial \D$, the radial limit $f^*(e^{i\theta}):=\lim\limits_{r \to 1^-}f(re^{i\theta})$ exists and belongs to $\partial \D$.
\end{definition} 
The map $f^*: \partial \D \to \partial \D$ is called the radial limit extension of $f$ and, from now on, we shall also write $f$ for $f^*$. Let $\Sigma$ denote the exceptional set of points where the radial limit does not exist or belong to $\partial \D$. If $\Sigma$ is non empty, then $f$ does not extend to a self map the closed unit disk $\overline{\D}$. 

We start with a classification of points in $\overline{\D}$. A point in $\D$ is called a regular value of $f$ if it has a neighbourhood on which all inverse branches of $f$ are well defined. A point which is not a regular value is called a singular value. We denote the set of singular values by $SV(f)$.

We say that a point $v \in \partial \D$ is a singularity of $f$ if $f$ can not be analytically continued to any neighbourhood of the point. The set of singularities of $f$ is denoted by $E(f)$. 
 The `standard' derivative is not well defined for points in the singular set. Instead, we normally use the Carathéodory angular derivative, which may exists at points in $E(f)$, as well as at all points in $\partial\D \setminus E(f)$. By convention, we set the angular derivative to be infinite at points where it doesn't exist. This convention is convenient as it is consistent with the explicit formula for the Carathéodory derivative, which is defined for all points on $\partial \D$-- see \cite[Theorem~4.15]{deriv}.  Moreover, outside of $E(f)$, the `standard' notion of derivative and the Carathéodory derivative agree, which is the only set on which we will ever need to give bounds on it. 

We now move on to giving some foundational results on inner functions.
The classical Denjoy–Wolff Theorem describes the long-term dynamics of the iterates $f^n$ on compact subsets of $\D$: 
\begin{theorem}[Denjoy-Wolff Theorem]\label{dw}
    Let $f$ be an inner function. Then exists a point $p \in \overline{\D}$ such that $\lim\limits_{n \to \infty}f^n(z)=p$ uniformly on compact subsets of $\D$. Furthermore, if 
$p \in \D$, then it is the unique fixed point of $f$ in $\D$, and $|f'(p)| < 1$. 
\end{theorem}
The point $p$ in the above theorem is called the Denjoy–Wolff point of $f$.  

A later result of Doering and Ma\~{n}\'{e} \cite[Corollary~1.5]{DM} states
 $p \in \D $ if, then $f$ preserves the harmonic measure. Recall that given $A \subset \D$, the harmonic measure of any $F \subset \partial \D$ with respect to $E$ is given by $$\omega(A,\D)(F)= \int_{F}\frac{1-|A|^2}{|A-z|^2}\,d\lambda(z) \, .$$
Note that in the case where $A=0$, we have $\omega(0,\D)(F)=\lambda(F)$, and so the harmonic measure is simply $\lambda$; i.e. normalised Lebesgue measure. This brings us to the following result.

\begin{lemma}[Löwner's Lemma] \label{low}
If $f : \D \to \D$ is an inner function, then its radial extension $f^*: \partial \D \to \partial \D$ preserves $\lambda$ if and only if $f$ is centred.
\end{lemma}  

Now that we have established a class of inner functions which preserve Lebesgue measure, we turn to exploring the expansion and mixing properties of inner functions. The measure-preserving property interacts naturally with the geometric expansion of $f$ on the boundary. We say that $f$ is expanding on $\partial\D$ if $|f'(z)|>1$ for all $z \in \partial\D$, and uniformly expanding on $\partial\D$ if  
$$
\inf_{z \in \partial\D} |f'(z)| > 1.
$$
The next lemma can be found in \cite[Corollary~4.16]{deriv}.
\begin{lemma}\label{inexp}
Let $f$ be a centred inner function. Then $f$ is uniformly expanding on $\partial \D$.
\end{lemma}
From Lemma~\ref{low} and Lemma~\ref{inexp}, we conclude that if $f$ is centred, then 
$
(\partial\D, \mathcal{B}, \lambda, f, d)
$
is a uniformly expanding, $\lambda-$measure-preserving dynamical system on $\partial \D$.
We say a centred inner function is contracting (in $\D$) if  $|f'(0)|=\alpha<1$.   Given a sequence $F_\N$ of centred inner function, we say $F_\N$ is contracting if $|f_n'(0)|=\alpha_n<1$ for all $n \in \N$, and, as mentioned in the introduction, uniformly contracting if $|f_n'(0)| \leq \alpha < 1$. For centred inner functions, $\alpha_n$ is the hyperbolic distortion of $f$ at $0$. More generally, we say a sequence of inner functions is contracting if $\sum_{n=1}^\infty (1-\alpha_n) = \infty $, in which case the hyperbolic distortion $\alpha_n$ may have a more complicated expression-- see \cite{BERFS} and the references within for details.
Therefore, centred inner functions always act as contractions on 
$\D$ and as expansions on $\partial \D$;  moreover, uniform contraction on 
$\D$ is equivalent to uniform expansion on $\partial \D$.

The next two key propositions, proven by Pommerenke in \cite{Pom81}, are in the non-autonomous setting. The first of these provides a useful zero-one law.
\begin{proposition} \label{contractpom}
    Let $F_\N$ be a sequence of contracting inner functions. 
If there are measurable subsets $L, L_n \subset \partial \D$ such that $L = F_n^{-1}(L_n)$ for all $n \in \N$ up to a set of (Lebesgue) measure zero, then $L$ has either full or zero measure on $\partial \D$.
\end{proposition}
 
In \cite[Lemma~3]{Pom81},  Pommerenke proved that a subclass of inner functions are exponentially mixing. Moreover, this mixing holds for compositions of a sequence of maps, not just iterates of a single map.
 \begin{proposition}[Exponential mixing for inner functions] \label{expmixcomp}
 Let $F_\N$ be a sequence of centred, uniformly contracting inner functions, so for all $n \in \N$ we have $|f'_n(0)| \leq \alpha <1$. Then, for some absolute $K > 0$ we have
 \begin{equation}\label{exppom}
   \left|\frac{\lambda(A\cap F_n^{-1}E)}{\lambda(E)}-\lambda(A)\right| \,  \le  \,  K \, e^{- n\tau}     \quad  \forall \, n \in \N,
\end{equation}
for all arcs $A \subset \partial\D$ and measurable sets $E$ in $\partial\D$ with $\lambda(E) > 0 $, where $\tau= \frac{1-\alpha}{84} $.
 \end{proposition}
\begin{remark}
The uniform contracting condition $|f_n'(0)| \leq \alpha<1$  ensures that the compositional map $F_n$ satisfies, $|F_n'(0)| \leq \alpha ^n$ for each $n \in \N$. In \cite{BERFS}, Benini-Evdoridou-Fagella-Rippon-Stallard study inner functions which are contracting (that is $|f_n'(0)|=\alpha_n <1$ for each $n \in \N$) rather than uniformly contracting, and show that a full measure shrinking target result follows if the sum of $r_n (1-\alpha_n)$ diverges, where $r_n$ is the sequence of radii in \eqref{defshrink}. 
 \end{remark}
Clearly, a system that satisfies \eqref{exppom} is uniformly mixing in the sense of Definition~\ref{ABunimix}. 
Therefore, we have established that a sequence of centred, uniformly contracting inner functions $F_{\N}$ gives rise to a non-autonomous system $(\partial\D, \mathcal{B},\lambda,F_{\N},d)$ which exponentially mixing and uniformly expanding on $\partial \D$.

\subsection{Proof of Theorem~\ref{ABthB}}\label{proofb}
The non-autonomous system which is the subject of Theorem~\ref{ABthB} is $(\partial \D, \mathcal{B},\lambda,F_\N,d)$, described in the last section. To prove this theorem, we must show that the system satisfies the hypotheses of Theorem~\ref{ABthA}. With Pommerenke's exponential mixing result (Proposition~\ref{expmixcomp}) this becomes a straightforward task.

\begin{proof}[Proof of Theorem~\ref{ABthB}]
     To apply Theorem~\ref{ABthA}, we must show that $F_\N$ is uniformly mixing. As discussed in Remark~\ref{fmpmix}, this is direct from Proposition~\ref{expmixcomp}. We now want to show that the asymptotic statement is governed by terms involving $r_n$ explicitly, not $\lambda(E_n)$. The measure $\lambda$ is $\delta$-Ahlfors regular with $\delta=1$, as clearly for all $x \in \partial \D$, $\lambda(B(x,r))=r/\pi$. Therefore, by \eqref{condmpt}, we have $\lambda(E_n)=r_n/\pi$, and this concludes the proof.
\end{proof}

\subsection{Background on base expansions}\label{backbase}
With Theorem~\ref{ABcf} in mind, we provide an appropriate overview of base expansions.
We start by considering maps of the form 
$$
t: [0,1] \to [0,1] \quad \text{ with } \quad  t(x) \equiv bx \mod 1 
$$
where $2\leq b \in \N$.  By induction, it is easy to show that $t^n(x) \equiv b^n x \mod 1$ for all $n$.  Iterating $t$ gives rise to a dynamical system on the unit interval;  $([0,1],\mathcal{B},\lambda,t, d_{e})$, where $d_e$ is Euclidean distance.
\\
Recall that the base $b$ expansion of $x \in [0,1)$ is given by
$$
x = \sum_{i=1}^{\infty} \frac{a_i}{b^i} \, , \  a_i \in \{0,\dots,b-1\} \, ,
$$
and, aside from the set of codings for which there exists $i_0$ such that $a_i=b-1 \: \forall \: i \geq i_0 \in \N$, this representation is unique. We write $x=[a_1,..,a_i,\dots]$, and can view $t$ as a left shift map:
$$t([a_1,\dots,a_i,..]) \, = \,  [a_2,\dots a_{i+1},\dots] \, .$$

 Notice that our condition on $2 \leq b \in \N $ is essential to establish the link to symbolic dynamics. We give an example to indicate that this ``link'' does not hold for all $b \in \R_{\geq2}$. 
\begin{example}
Let $b= \sqrt{5}$, and consider $t(x) \equiv \sqrt{5}x \mod 1$. Consider $x=1/2$. Then $t^2(x)\equiv5/2\mod 1$, but $t(t(x)\mod 1)\mod 1$ is irrational.
\end{example}
We can explicitly relate the map $t$ to the inner function $g(z)=z^b$ by the map $\phi : [0,1] \to \partial \D \, , \phi(\theta)=e^{2 \pi i \theta}$ which induces an isomorphism of dynamical systems:
$$([0,1],\mathcal{B},\lambda,t, d_{e}) \cong (\partial \D,\mathcal{B}, \lambda,g, d) \, .$$ 
This isomorphism is used to translate between the study of inner functions and base expansions-- see, for example, \cite[Theorem~4.5]{FMP}. 

We turn to the corresponding set up for compositions of base expansions. The sequence of maps $t_i(x)\equiv b_i x \mod 1$ gives rise to the dynamical system $([0,1], \mathcal{B}, \lambda, T_{\N}, d_{e})$. Similarly, the sequence
 $g_i(z)=z^{b_i}, \, b_i \in \N_{\geq 2}$ gives rise to the dynamical system $(\partial \D, \mathcal{B},\lambda, G_{\N}, d)$. It is not hard to check that these are indeed isomorphic: 
 \begin{eqnarray}\label{iso}
     ([0,1], \mathcal{B}, \lambda, T_{\N}, d_e) \cong (\partial \D, \mathcal{B}, \lambda, G_{\N}, d) \, .
 \end{eqnarray}
 As described in the introduction, the multibase or `Cantor series' expansion of $x \in [0,1)$ with respect to $\{b_i\}_{n \in \N}$, $b_i \in \N_{\geq 2}$ is the unique expansion:
$$
x = \sum_{i=1}^{\infty} \frac{d_i}{b_1\dots b_i} \, , \  d_i \in \{0,\dots,b_i-1\}  \, ,
$$
where $d_i \neq b_i -1$ for all $i\geq i_0$, and is denoted by $x= [d_1,..,d_i,\dots]$. 
In contrast to the fixed base case $b$, the transformation $T_n$ can no longer be interpreted as a simple left-shift of the digit string; that is to say:
$$
T_1([d_1,\dots,d_i,\dots]) \, \neq \, [d_2,\dots,d_{i+1},\dots] \, .
$$
Indeed, in the multibase setting, we have $$T_n( [d_1,\ldots,d_n,\ldots] ) = \frac{ d_{n+1} }{ b_{n+1} } + \frac{ d_{n+2} }{ b_{n+1} b_{n+2} } + \cdots \, .$$
Therefore, the map $T_n$ acts as a left shift not only on the digits but also on the sequence of bases itself; it replaces the origional base sequence $\{b_i\}_{i \in \N}$ by the shifted sequence $ \{b_{n+i}\}_{i \in \N}$. As in the autonomous setting, the sequence of bases was the constant sequence $b_n=b$ for all $n \in \N$, this shift in the bases is invisible; there is no effective change in the underlying expansion mechanism. However, as the proof of Theorem~\ref{ABcf} reveals, in the non-autonomous setting the shift of the base sequence is exactly what restricts the result to sequences of zeros, rather than to general digit sequences.

Finally, we conclude the section by mentioning \cite[Theorem~1.1]{Sun}, in which Sun and Cao obtain a zero-one criterion (without asymptotic information) for the shrinking target set associated to the multibase expansion with a fixed $x_0$.
\medskip

\subsection{Proof of Theorem~\ref{ABcf}} \label{apply}
 Theorem~\ref{ABcf} will be deduced as an application of Theorem~\ref{ABthB}.
\begin{proof} [Proof of Theorem~\ref{ABcf}]
Let $G_\N$  be a sequence of centred, uniformly contracting inner functions of the form $g_i=z^{b_i}$, and $G_n=g_n \circ\dots\circ g_1 $. Then, by Proposition~\ref{expmixcomp}, $(\partial \D, \mathcal{B}, \lambda, G_{\N}, d)$ is an exponential mixing, $\lambda$-measure-preserving  system. The isomorphism given in \eqref{iso} implies that $([0,1], \mathcal{B},\lambda, T_{\N}, d_e)$ with $t_i\equiv b_i \mod 1$ is also exponentially mixing. Hence, this system satisfies all the conditions of Theorem~\ref{ABthB}.
   
We now illustrate the link between the shrinking target set and the digit set $D(\xi)$. Given a point $x=[d_1,\dots,d_n,\dots] \in [0,1)$, the map $T_n$ acts as follows on $x$:
$$T_n(x) = \frac{d_{n+1}}{b_{n+1}} \ + \ \frac{d_{n+2}}{b_{n+1}b_{n+2}} \ + \ \dots \ + \ \frac{d_{n+\xi(n)}}{b_{n+1}\dots b_{n+\xi(n)}} \ + \ \frac{d_{n+1+\xi(n)}}{b_{n+1}\dots b_{n+1+\xi(n)}} \ + \ \dots \, $$ 
To force $ d_{n+1} = d_{n+2} = \cdots = d_{n+\xi(n)}=0$, it suffices to take 
$$
|T_n(x)| < r_n \quad \text{where} \quad   r_n:= \frac{ 1 }{ b_{n+1} \cdots b_{n+\xi(n)+1} } \, .$$ This inequality is equivalent to the statement that \begin{eqnarray}\label{ineq1}\quad  T_n(x) \  \in  \ B\left(0, r_n\right) \, .
\end{eqnarray} 
For each $n \in \N$, by choosing $\xi$ so that
\begin{eqnarray}\label{cond1}
    \frac{b_{n+1}}{b_{n+\xi(n)+2}\dots b_{n+\xi(n+1)+2}}\leq 1 \, ,
\end{eqnarray} 
 we can ensure that the sequence $\{r_n\}_{n \in \N}$ is decreasing.
Now observe that \eqref{ineq1} is equivalent to $x \in E_n$, where $E_n$ is the building block set associated to the shrinking target set $\mathcal{W}(T_\N,0,\{r_n\})$.
Thus, it follows that $$x \in \mathcal{W}(T_\N,0,\{r_n\}) \implies d_{n+1}=...=d_{n+\xi(n)}=0 \text{ \ for infinitely many \ } n \in \N \, $$ and, in particular, Theorem~\ref{ABcf} follows from a direct application of Theorem~\ref{ABthB}.
\end{proof}
\medskip
As mentioned in the introduction, previous works such as \cite{mance} explore the typicality of normal numbers for the multibase expansion. Therefore, the more `interesting' case of Theorem~\ref{ABcf} is when $D(\xi)$ consists of arbitrarily long strings of zeros. We wish to analyse when this case occurs on a set of full measure. If the sequence $\{b_n\}_{n \in \N}$ is bounded, say $b_n \leq K$ for all $n \in \N$, then we can always find an increasing, unbounded $\xi$ that satisfies the divergent sum condition. Indeed notice that 
$$
r_n = \frac{1}{b_{n+1}...b_{n+1+\xi(n)}} \geq \frac{1}{K^{\xi(n)}} \, .
$$
Picking $\xi(n) \approx \log_K(n) \implies K^{\xi(n)}  \approx K^{\log_K(n)} = n$, and so we are done.
This motivates us to ask the following:
 \begin{question}
Given a sequence of bases $\{b_n\}_{n \in \N}$ increasing to infinity, is it possible for the sum $\sum_{n=1}^\infty\frac{1}{b_{n+1}\dots b_{n+1+\xi(n)}}$ to diverge if $\xi$ is an increasing and unbounded function?
 \end{question} 
 We first note that if we impose the condition that $\{b_n\}_{n \in \N}$ is an increasing sequence, then \eqref{cond1} is satisfied independently of the choice of $\xi$.
The answer then depends on the growth rate of $\{b_n\}_{n \in \N}$, as we illustrate in the following examples. The first shows that if $\{b_n\}_{n \in \N}$ is an increasing sequence of power law, then the answer to the question is in the negative.
\begin{example}
    Let $b_n \approx  n^\epsilon$ for any $\epsilon > 0$. Then we shall show that $\sum_{n=1}^\infty r_n = \infty$ implies $\xi(n) \leq K$. Suppose, for contradiction, that we could find an unbounded, increasing $\xi(n)$ which made the sum diverge. Then there exists some $n_0$ such that for all $n \geq n_0$, we have $\xi(n) \ge \lceil 2/\epsilon \rceil$. Hence,
    $$r_n=\frac{1}{b_{n+1}\dots b_{n+1+\xi(n)}} \lesssim \frac{ 1 }{ (n^\epsilon)^{2/\epsilon} }=\frac{1}{n^2} \, .$$ 
    Clearly, this contradicts the divergence of the sum of $r_n$, and hence such a $\xi$ does not exist.
\end{example}
However, if $\{b_n\}_{n \in \N}$ is an increasing sequence of logarithmic growth, then we can answer the question in the affirmative:
\begin{example}
Let $b_n \approx \log(n)$. Then we can choose $\xi(n) \approx \big(\log(n)\big)^{1- \epsilon}$ for some $0 < \epsilon <1$ such that $\sum_{n=1}^\infty r_n = \infty$. Indeed, observe that 
$$
\frac{1}{b_{n+1}\dots b_{n+1+\xi(n)}} \gtrsim \frac{ 1 }{ \Bigl( \log( n + (\log(n))^{1-\epsilon} + 1 ) \Bigr)^{ (\log(n))^{1-\epsilon} } }
 \, .
$$
Notice that 
$$
\Big(\log\big(n + (\log(n))^{1-\epsilon}+1\big)\Big)^{1-\epsilon} = \exp\Big((\log(n))^{1-\epsilon}\log\log \big( n + (\log(n))^{1- \epsilon} +1   \big) \Big) \, .
$$
Then we can find some $n_0$ such that for all $n \geq n_0$, 
$$
\log\log \big( n + (\log(n))^{1- \epsilon} +1   \big) (\log(n))^{1-\epsilon} \leq \frac{1}{2} \log(n) \, ,
$$
and so
$r_n \geq \frac{1}{\sqrt{n}}$. This proves that the sum diverges as desired.
\end{example}
The above analysis is centred around bases $\{b_n\}_{n \in \N}$ which are increasing and unbounded. It would be interesting to consider the situation where $\{b_n\}_{n \in \N}$ is unbounded, but not increasing.
As previously mentioned, for single base $b$ expansions, the result of \cite[Theorem 5]{FMP} is similar in flavour to Theorem \ref{ABcf}. However, their theorem applies to digit blocks that are not necessarily zero. In the context of the proof of Theorem~\ref{ABcf}, this would correspond to allowing the point $x_0$ to be arbitrary rather than fixing $x_0=0$. As discussed in Section \ref{backbase}, the result is more general for the single base expansion because the map $t(x)=bx\mod 1$ acts as a left shift on the digit expansion while leaving the constant base sequence unchanged. Using Theorem \ref{ABthB}, we can adapt their argument to obtain an asymptotic version of their result, with an explicit error term.
\medskip

\section{Proof of non-autonomous recurrence laws} \label{nonrecres}
We turn to proving recurrence results. In this section, we prove Theorem~\ref{ABconv} and Theorem~\ref{ABconvin}, which each give a zero measure law for the recurrence set associated to a non-autonomous system. The following lemma will be crucial in the proof of all of our recurrence results. 
    \begin{lemma}\label{trig}
    Let $B=B(x_0,r) \subset X$. Then for any $n \in \N$ with $r_n >r$, and any $E \subset B$, 
    $$
    E \cap T_n^{-1}B(x_0, r_n -r) \; \subset \; E \cap A_n \; \subset \;  E \cap T_n^{-1}B(x_0, r_n +r) \, .
    $$
\end{lemma}
\begin{proof}[Proof of Lemma~\ref{trig}]
Let $x \in E \cap T^{-1}_nB(x_0,r_n-r)$. Then 
$d(x,x_0)< r_n$ and $d(x_0, T_n(x))< r_n-r$.
By the triangle inequality,
$$
d(x,T_n(x)) \leq d(x,x_0) + d(x_0,T_n(x)) < r + r_n - r = r_n \, .
$$
Hence,  $x \in A_n$ and so we have our first inclusion. Let $y \in E \cap A_n$. Then $d(x,x_0)<r$ and $d(y,T_n(y))<r_n+r$. By the triangle inequality, 
$$
d(x_0,T_n(y)) \leq d(y,x_0) + d(y,T_n(y))<r_n + r \, .
$$
Hence, $y \in T_n^{-1}B(x_0,r_n+r)$, and so we have the second inclusion.
\end{proof}
This lemma gives us a way of locally estimating the measure of the recurrence building block set $A_n$ via preimages of balls. With Lemma~\ref{trig} in hand, we begin the proofs of our non-autonomous zero measure laws.
\subsection{Proof of Theorem~\ref{ABconv}}  
Let $(X,\mathcal{B},\mu,T_{\N},d)$ be as in Theorem~\ref{ABconv}.
While the overall structure of the proof of Theorem~\ref{ABconv} mirrors \cite[Theorem~1.3]{HLSW} (which proves the result in the autonomous case) several estimates must be re-established in the broader, non-autonomous context. 
We actually will prove the following proposition, from which Theorem~\ref{ABconv} easily follows. Throughout, the constants $c_1,c_2, \delta$ are the constants arising from the Ahlfors regularity of $\mu$-- see \eqref{defahl}.
\begin{proposition}\label{maincon1iff} Let $(X,\mathcal{B},\mu,T_{\N},d)$ be as above. Given a positive sequence $\{r_n\}_{n \in \N}$, we have
   \begin{eqnarray}\label{maincon1}
    \sum_{n=1}^{\infty} r_n^\delta < \infty \iff  \sum_{n=1}^\infty \mu(A_n) < \infty \, ,
\end{eqnarray}
    where $\delta$ is the constant from the Ahlfors regularity of $\mu$ (see \eqref{defahl}).
\end{proposition}
We delay the proof of Proposition~\ref{maincon1iff} until after the proof of Theorem~\ref{ABconv}.
\begin{proof}[Proof of Theorem~\ref{ABconv}]
By assumption, $\sum_{n=1}^\infty r_n^\delta$ converges, and so the forward implication of Proposition~\ref{maincon1iff} implies that $\sum_{n =1}^\infty \mu(A_n)$ also converges. Hence, on applying Lemma~\ref{cbc} we have $\mu(\limsup\limits_{n \to \infty} A_n) =0$. As $\limsup\limits_{n \to \infty} A_n =\mathcal{R}(T_\N,\{r_n\})$, we have our result.
\end{proof}
Therefore, most of the work to prove Theorem~\ref{ABconv} is in establishing the forward implication of \eqref{maincon1}. We now prove Proposition~\ref{maincon1iff}, and so in the process, also this implication.
\begin{proof}[Proof of Proposition~\ref{maincon1iff}]
We start by establishing local estimates for $\mu(A_n)$.
\begin{lemma}\label{lban} Let $B:=B(x_0,\epsilon r_n)$, with $0<\epsilon<1$. Then
   $$ c_1(r_n(1 - \epsilon))^\delta \big(\mu(B\big) - \phi(n)\big) \, \leq \, \mu(B \cap A_n) \, \leq \, c_2(r_n(1 + \epsilon))^\delta \big(\mu(B\big) + \phi(n)\big) \, ,
   $$
   where $\phi(n)$ is as in Definition~\ref{ABunimix}.
\end{lemma}
\begin{proof}
 Let $B^{\pm}:=B(x_0, (1\pm\epsilon)r_n)$. Then, on using Lemma~\ref{trig} and uniform mixing, we have \begin{align}
    \mu(B^-) \big(\mu(B\big) - \phi(n)\big)\; \leq \; \mu(B \cap A_n) \; \leq \;  \mu(B^+) \big(\mu(B\big) + \phi(n)\big) \, .
\end{align} Then, by Ahlfors regularity, we have $c_1(r_n(1\pm \epsilon))^\delta\leq \mu(B^\pm)\leq c_2(r_n(1\pm \epsilon))^\delta$ and so we are done.
\end{proof}
To estimate $\mu(A_n)$, we will use the $5$B-Covering Lemma along with Lemma~\ref{lban}.
\begin{lemma}\label{anbd}
Let $0<\epsilon<1$. Then for all $n$,
 $$      \frac{c_1(1-\epsilon )^\delta}{5^\delta} \big( r_n^\delta - \frac{\phi(n)}{c_2 \epsilon^\delta}\big)
   \; \leq \; \mu(A_n) \;  \leq  \; \frac{ c_2(1+\epsilon )^\delta}{ \epsilon^\delta}\big(  r_n^\delta+\frac{\phi(n)}{c_1 \epsilon^\delta} \big).
 $$
\end{lemma}
\begin{proof}
     We can find a collection of balls $\{B(x,\epsilon r_n): x\in X\} $ which covers $X$. Then (as any compact metric space is separable) we  can apply $5$B-Covering Lemma, to find a countable, disjoint subcollection of these balls 
    $
    \{B(x_j,\epsilon r_n)\}_{j \in \mathcal{J}}
    $
    such that 
    \begin{eqnarray}\label{inc}
        \bigcup_{j \in \mathcal{J}}B(x_j,\epsilon r_n) \subset X \subset \bigcup_{j \in \mathcal{J}}B(x_j,5\epsilon r_n) \, .
    \end{eqnarray}
We now use Ahlfors regularity along with \eqref{inc} to obtain upper and lower bounds on $\#\mathcal{J}$. For upper bounds:
$$
1=\mu\Big( \bigcup_{j \in \mathcal{J}}B(x_j,5\epsilon r_n)\Big) \geq \sum_{j \in \mathcal{J}} \mu\big(B(x_j, \epsilon r_n)\big) \geq 
\sum_{j \in \mathcal{J}}c_1 (\epsilon r_n)^\delta \geq c_1(\epsilon r_n)^\delta \sum_{j \in \mathcal{J}}1  \, .
$$ 
Similarly, for lower bounds: 
$$
1=\mu\Big( \bigcup_{j \in \mathcal{J}}B(x_j,5\epsilon r_n)\Big) \leq \sum_{j \in \mathcal{J}} \mu\big(B(x_j, 5\epsilon r_n)\big) \leq 
\sum_{j \in \mathcal{J}}c_2 (5\epsilon r_n)^\delta \, \leq c_2(5\epsilon r_n)^\delta \sum_{j \in \mathcal{J}}1  \, .
$$
Therefore, $\frac{1}{c_1(5\epsilon r_n)^\delta}\leq \#\mathcal{J}\leq\frac{1}{c_1(\epsilon r_n)^\delta}$. We now observe that
$$
\bigcup_{j \in \mathcal{J}}(B(x_j,\epsilon r_n) \cap A_n) \subset A_n \subset \bigcup_{j \in \mathcal{J}}(B(x_j,5\epsilon r_n) \cap A_n) \, ,
$$
and so, by disjointness,
\begin{eqnarray}\label{eqnbound}
     \sum\limits_{j \in \mathcal{J}}\mu \big(B(x_j,\epsilon r_n) \cap A_n\big) \leq \mu(A_n) \leq \sum\limits_{j \in \mathcal{J}} \mu (B(x_j,5\epsilon r_n) \cap A_n)\, .
\end{eqnarray}
 We apply Lemma~\ref{lban} to each term appearing in the summations of \eqref{eqnbound}, and so our bounds become
\begin{align*}
\#\mathcal{J}\,c_1\,r_n^\delta(1-\epsilon )^\delta
   \left(\mu(B(x,\epsilon r_n)) - \phi(n)\right)
   &\,\leq\, \mu(A_n) \, \leq \,
   \#\mathcal{J}\,c_2\,r_n^\delta(1+\epsilon )^\delta
   \left(\mu(B(x,\epsilon r_n)) + \phi(n)\right).
\end{align*}

Then, by Ahlfors regularity once more and substituting in the bound on $\# \mathcal{J}$ we have:
 \begin{align*}
   \frac{c_1(1-\epsilon )^\delta}{5^\delta} \left( r_n^\delta - \frac{\phi(n)}{c_2 \epsilon^\delta}\right)
    \leq   \mu(A_n)  \leq  \frac{ c_2(1+\epsilon )^\delta}{ \epsilon^\delta}\left(  r_n^\delta+\frac{\phi(n)}{c_1 \epsilon^\delta} \right).
\end{align*}
\end{proof}
    Hence, taking $\epsilon=1/2$ in Lemma~\ref{anbd}, and summing over all $n \in \N$,  we have $$\sum\limits_{n=1}^\infty r_n^\delta  - \sum\limits_{n=1}^\infty \phi(n) \lesssim_{\delta} \sum\limits_{n=1}^\infty \mu(A_n) \lesssim_{\delta} \sum\limits_{n=1}^\infty r_n^\delta  + \sum\limits_{n=1}^\infty \phi(n) \, \, .$$ 
    Here, $\lesssim_{\delta}$ means the inequality holds with some constant depending on only $\delta$.
    As $ \sum\limits_{n=1}^\infty \phi(n)$ converges, we have our result.
\end{proof}

\subsection{Proof of Theorem~\ref{ABconvin}}
Recall that Theorem~\ref{ABconvin} concerns the system $(\partial \D, \mathcal{B},\lambda , F_\N, d)$, where $F_\N$ is a sequence of centred, uniformly contracting inner functions. 
The proof is a direct corollary of Theorem~\ref{ABconv} applied to this setting.
 \begin{proof}[Proof of Theorem~\ref{ABconvin}]
As established in Section~\ref{preres}, the conditions on $F_\N$ ensure the system is exponentially mixing (Proposition~\ref{expmixcomp}). As $\lambda$ is normalised Lebesgue measure, it is Ahlfors regular. Hence, we obtain our theorem as a corollary of Theorem~\ref{ABconv}.
  \end{proof}
 We provide an alternative proof of Theorem~\ref{ABconvin}, as we can simplify some steps due to the one-dimensional nature of the  system, and by using properties of inner function. For example, the following proof will not rely on $5$B-Covering Lemma. Moreover, we will need the explicit local bounds which we will obtain for $\lambda(A_n)$ in the proof of Theorem~\ref{ABrecocif}.
\begin{proof}[Alternate proof of Theorem~\ref{ABconvin}]
As before, we start by estimating $\lambda(A_n)$ locally.
\begin{lemma}\label{ban}
Fix some $n \in \N$, let $0 <\epsilon \leq 1/2$ and $B:=B(x_0,\epsilon r_n) \subset \partial \D$. Then 
$$
\frac{r_n(1-\epsilon)}{\pi}\big(\lambda(B) - e^{-n\tau} \big) \: \leq \: \lambda(B \cap A_n) \: \leq \:
\frac{r_n(1+\epsilon)}{\pi}\big(\lambda(B)+e^{-n\tau} \big)  \, .
$$
    
\end{lemma}
\begin{proof}[Proof of Lemma~\ref{ban}]
Let $B^\pm:= B(x_0,r_n(1\pm\epsilon))$. We apply Lemma~\ref{trig}, setting $T_n=F_n$ and $E=B$, before using exponential mixing (Proposition~\ref{expmixcomp}):
\begin{eqnarray*}
    \lambda\big(B \cap F_n^{-1}B^- \big) \quad &\leq& \quad  \lambda \big( B \cap A_n \big) \quad  \leq \quad  \lambda\big(B \cap F_n^{-1}B^+ \big) \\  
    \lambda\big(B^- \big)\big(\lambda(B) - e^{-n\tau} \big)\quad  &\leq& \quad \lambda\big( B \cap A_n \big) \quad \leq \quad \lambda(B^+)(\lambda(B)+e^{-n\tau})
    \\ \frac{r_n(1-\epsilon)}{\pi}\big(\lambda(B) - e^{-n\tau} \big)\quad  & \leq& \quad  \lambda\big( B \cap A_n \big) \quad  \leq \quad \frac{r_n(1+\epsilon)}{\pi}\big(\lambda(B)+e^{-n\tau} \big) \, .
\end{eqnarray*}
\end{proof}
We can now use Lemma~\ref{ban} to estimate $\lambda(A_n)$ in a way that avoids using the $5$B-Covering Lemma.
\begin{lemma}\label{an} Given $0<\epsilon<\frac{1}{2}$, for all $n \in \N$
  $$ \frac{r_n(1-\epsilon)}{\pi} - \big(\frac{1}{\epsilon}-1\big) e^{-n\tau}\leq  \lambda(A_n) \leq \frac{r_n(1+\epsilon)}{\pi}+ 2\big(\frac{1}{\epsilon}+1\big)e^{-n\tau} $$
\end{lemma}
\begin{proof}[Proof of Lemma~\ref{an}]
 Given $0<\epsilon< \frac{1}{2}$, we tile $\partial \D$ with disjoint balls $B$ of radius $r:=\epsilon r_n$ except for the last ball, which may have a smaller radius to preserve disjointness. Denote this family by $\mathscr{F}$, and so $\# \mathscr{F}=  \lceil \frac{\pi}{\epsilon r_n} \rceil$. Notice that 
\begin{eqnarray}\label{fbound}
 \frac{\pi}{\epsilon r_n} \leq   \#\mathscr{F} \leq \frac{2\pi}{\epsilon r_n} \, .
\end{eqnarray}
Given this cover, it is clear that  $$\lambda(A_n) = \sum_{B \in \mathscr{F}}\lambda(B \cap A_n) \, ,$$  and so
\begin{eqnarray*}
   \frac{1}{\pi}\sum_{B \in \mathscr{F}} r_n(1-\epsilon)\big(\lambda(B)-e^{-n\tau} \big) \; \leq \; \lambda(A_n) \; &\leq &\; \frac{1}{\pi}\sum_{B \in \mathscr{F}}r_n(1+\epsilon)\big(\lambda(B)+e^{-n\tau} \big)  \\ \frac{1}{\pi}r_n(1-\epsilon)\ - \frac{\#\mathscr{F}}{\pi}r_n(1-\epsilon)e^{-n \tau} \;  \leq \; \lambda(A_n) \; &\leq& \;
    \frac{1}{\pi}r_n(1+\epsilon)+ \frac{\#\mathscr{F}}{\pi}r_n(1+\epsilon) e^{-n\tau}
    \\  \frac{r_n(1-\epsilon)}{\pi} - \big(\frac{1}{\epsilon }-1\big) e^{-n\tau} \; \leq \; \lambda(A_n) \; &\leq& \; \frac{r_n(1+\epsilon)}{\pi}+ 2\big(\frac{1}{\epsilon}+1\big) e^{-n\tau}\, ,
\end{eqnarray*}
    where for the last inequality, we used \eqref{fbound}. 
\end{proof}

Finally, we are in a position relate the summability of $\lambda(A_n)$ to that of  $r_n$.
\begin{lemma}\label{suman}
Let $\gamma:=\sum_{n=1}^\infty e^{-n \tau}$. Then 
    $$
    \frac{1}{2\pi}\sum_{n=1}^\infty r_n- \gamma \leq  \sum_{n=1}^\infty \lambda(A_n) \leq \frac{3}{2\pi} \sum_{n=1}^\infty r_n +  6\gamma\, .
    $$
\end{lemma}
\begin{proof}[Proof of Lemma~\ref{suman}]
    Summing over Lemma~\ref{an},
    \begin{eqnarray*}
     \frac{(1-\epsilon)}{\pi}\sum_{n=1}^\infty r_n -\big(\frac{1}{\epsilon}-1\big)\sum_{n=1}^\infty e^{-n\tau}  \, \leq \,  \sum_{n=1}^\infty \lambda(A_n) \, \leq \, \frac{(1+\epsilon)}{\pi}\sum_{n=1}^\infty r_n +\big( \frac{2}{\epsilon}+2\big)\sum_{n=1}^\infty e^{-n\tau} \, .
    \end{eqnarray*}
    Choosing $\epsilon =1/2$ gives us 
    the desired result.
    \end{proof}
We finish the proof of Theorem~\ref{ABconvin} by applying the First Borel-Cantelli Lemma, and using the fact that, by Lemma~\ref{suman}, $\sum_{n=1}^\infty r_n < \infty \implies \sum_{n=1}^\infty \lambda(A_n) < \infty $.
\end{proof}

\medskip

\section{Autonomous recurrence background and proof}\label{autrecres}
In this section, we start by providing preliminaries on the class of one-component inner functions. In particular, we demonstrate that they admit a Markov partition on $\partial \D$, and prove this partition satisfies certain regularity properties. We then finish with the proof of Theorem~\ref{ABrecocif}.
\subsection{Background on one-component inner functions}\label{backocif]}
Centred, one-component inner functions were first studied by Cohn \cite{cohn} due to their operator-theoretic applications and connection with embedding theorems. For us, they are a natural candidate class of inner functions because they admit a Markov partition -- see Section~\ref{Markov}. In this section we record several key properties, including the location of their singular values and the relation between their exceptional and singular sets on the unit circle. As our primary motivation lies in recurrence phenomena rather than the intrinsic theory of one-component functions, we restrict attention to those properties most relevant for recurrence.
\begin{definition}\label{defocf}
    We say an inner function $f$ is one-component if there exists $0<r<1$ such that $\{z \in \D: |f(z)| < r\}$ is connected.
\end{definition}
\begin{remark}\label{altdefocf}
We also have the alternative characterisation, which will be particularly useful later. One-component inner functions have their singular values compactly contained in $\D$. This means there exists a $\rho$ such that for some $0<\rho<1$, the annulus centred at $0$ with inner radius $\rho$ and outer radius $1$, denoted $A(0;\rho,1)$, does not contain any singular values. Notice that by the Schwarz Reflection Principle, we can extend the area which avoids singular values of $f$ to $A_\rho:=A(0;\rho,1/\rho)$. Here, we are considering the maximum meromorphic extension of the map $f: \hat{\C}\setminus \Sigma \to \hat{\C}$.
\end{remark} 

The following lemma is part of a broader equivalences given in \cite[Theorem~7.2 (d)]{urb}. 
 \begin{lemma}\label{svep}
    Let $f$ be a one-component inner function. Then the set of exceptional points $\Sigma$ of $f$ coincides with the set of singularities on the unit circle of $f$. In particular, $\Sigma$ is closed. Furthermore, for  $z \in \partial \D \setminus \Sigma$, we have $f'(z) \to \infty$ as $z \to \Sigma$.
 \end{lemma}
Thus, for such functions we have $E(f)=\Sigma$. It follows that the angular derivative only fails to exist on $\Sigma$, where by convention it is infinite. Consequently, for centred, one-component inner functions, we may regard $f'$ as the `standard' derivative outside of $\Sigma$. 

For finite degree maps, every regular value has exactly $\deg(f)$ preimages. This cannot be directly extended to infinite degree maps without clarifying if the degree of the map is countable or not. In fact, centred, one-component inner functions are essentially countable to one, meaning that for $\lambda$ almost every $z \in \partial \D$, $\{f^{-1}(z)\}$ is at most countable \cite[Corollary~3.3]{urb}.

Finally, following from Remark~\ref{altdefocf}, we have 
\begin{lemma}\label{annul}
 Let $f$ be a centred, one-component inner function. Then every branch of $f^{-1}$ is well defined in a ball centred on $\partial \D$, which is contained in $A_\rho$. Furthermore, as $$f^{-1}A_\rho \subset A_\rho \, , $$  for all $n \in \N$, all branches of $f^{-n}$ are well defined on the same ball in $A_\rho$.
\end{lemma}
\begin{proof}
    By the Schwarz lemma, $f^{-1}A_\rho \subset A_\rho$, and by \cite[Lemma~7.1]{urb}, the map $f: \D \to \D$ is a covering map on this annulus. Hence, every branch of $f^{-1}$ is well defined in $A_\rho$. Furthermore, as $f^{-n}A_\rho \subset A_\rho$, the same is true for $f^{-n}$.
\end{proof} 
In fact, more is true. The map $f: \D \to \D$ is a covering map on $A(0;\rho,1)$ and the extension of the map is a covering map on $A_\rho$.

Finite degree one-component inner functions are, in fact, equivalent to a ubiquitous family of functions, as the following lemma from \cite{CM} illustrates.
\begin{lemma}\label{foc}
Let $f$ be an inner function. Then $f$ is a finite degree one-component inner function if and only if $f$ is a finite Blaschke product.
\end{lemma}

With Lemma~\ref{foc} in mind, we recall the definition of a Blaschke product.
Recall that, for a given sequence of $\{a_n\}$ in $\D$ which satisfy the Blaschke condition \begin{align*}
     \sum_{n} \big(1-|a_n|\big)<\infty \, , 
\end{align*}
the Blaschke product with zeros at $\{a_n\}$ is defined by
\begin{align}\label{defblas}
 b(z):= \prod\limits_{n} \frac{|a_n|}{a_n} \frac{a_n-z}{1-\overline{a_n}z} \,  \quad \text{ \ for \ } \,  z \in \D \, .
\end{align}
Each zero is repeated according to its multiplicity, and the convention $\frac{|a_n|}{a_n}=1$ is used when $a_n=0$.
When the product in \eqref{defblas} finite, we say $b$ is a finite Blaschke product. 
Any analytic function $f: \D \to \D$ which can be extended to a continuous function $f:\overline{\D} \to \overline{\D}$ is actually a finite Blaschke product. We stress that this means the radial limit of $f$ exists everywhere on $\partial \D$ (i.e. $\Sigma=\emptyset$). Therefore, $f$ is a self map of $\overline{\D}$, and hence an inner function. We summarise the key consequences of this fact in the following lemma.
\begin{lemma}\label{blasprop}
Let $b(z)$ be a degree $n$ finite Blaschke product. Then 
\begin{enumerate}
    \item {$b$ is continuous on $\partial \D$};
    \item {$|b|=1$ on $\partial \D$};
    \item {$b$ has $n-1$ fixed points on $\partial \D$, all of which are repelling}.
\end{enumerate}
\end{lemma}

Before proving Theorem~\ref{ABrecocif}, we need the following background.
\subsection{Markov Partition for one-component inner functions}\label{Markov}
To date, all recurrence results rely on the dynamical system in question exhibiting a partition with certain regularity properties. 
We discuss Markov partitions for the autonomous measure preserving dynamical system $(X,\mathcal{B},\mu,t,d)$. Informally, a Markov partition decomposes the space $X$ into at most countably many connected components $\{U_i\}_{i \in \mathcal{I}}$ with minimal overlap. These generate each `level' of the partition such that the restriction of the map on each level $n$ piece is a homeomorphism mapping the piece to a union of partition elements. If the collection of $U_i$ is finite, we say that $t$ admits a finite partition. If the collection is (countably) infinite, we say $t$ admits an infinite Markov partition.    
 More precisely, we define level $n$ of the partition to be 
 \begin{eqnarray}\label{cylset}
     \mathcal{F}_n:=\{ U_{i_0} \cap T^{-1}U_{i_1} \cap \dots. \cap T^{-(n-1)}U_{i_{n-1}}: i_0,..,i_{n-1} \in \mathcal{I}\} \, .
 \end{eqnarray}
 A connected component $J_n \in \mathcal{F}_n$ is called a cylinder set of level $n$. For more details on Markov partitions and an explicit construction of one for some inner functions, see \cite{fernandez2007}.
 
 We now demonstrate that every centred, one-component inner function $f$ admits such a partition. The construction splits naturally into two cases depending on whether $\deg(f)<\infty$ or not. 
By Lemma~\ref{foc}, in the former case, $f$ is  a finite Blaschke product $b$, say $\deg(b)=N$.  First, by Lemma~\ref{blasprop}, $b$ has an unramified, repelling fixed point $x \in \partial \D$. As all points on $\partial \D$ are regular:
$$\{b^{-1}(x)\}= \{x=x_1,\dots,x_N \} \, \subset \partial \D$$
counting multiplicity.
 These preimages divide the unit circle into $N$ complimentary open arcs $\{U_i\}$ with end points at $x_i,x_{i+1}$ for $i \in \{1,..,N-1\}$ and $U_N$ with end points at $x_{N},x_0$. It is clear that $\partial \D  = \bigcup_{i=1}^N \overline{U_i}$, and also $b(U_i)= \partial \D \setminus \{x\}$. Crucially, for each $i$, $b|_{U_i}$ is a continuous, injective map. Uniform expansion on $\partial \D$ implies that $\arg(f)$ is monotonic on each $U_i$.
As described in \eqref{cylset}, we can generate level $n$ of the partition through these $U_i$. Furthermore, it is clear that we can encode any point in $ x \in \partial \D$ with an infinite itinerary $(i_1,\dots,i_n,\dots)$ with $i_j \in \{1,\dots,N\}$, recording which cylinder set $x$ lies in at any given level. This itinerary is unique for all but a countable set of points in $\partial \D$. 
\\

In \cite[\S8]{urb},  Ivrii-Urba\`nski construct a Markov partition for infinite degree one-component inner functions which generalises the above method for finite Blaschke products. This construction is more delicate because now $\Sigma \neq \emptyset$.  Therefore, $\Sigma$ already partitions $\partial \D$ into arcs $\{I_i\}$. By Lemma~\ref{svep}, as $f$ is one-component, its associated exceptional set $\Sigma$ coincides with the singular set $E(f)$ and so is closed. We can decompose the unit circle into open arcs $\partial \D \setminus \Sigma = \{I_i\}_{i=1}^\infty$. By Remark~\ref{altdefocf}, we also know that all points in $\partial \D $ are regular.
We first must show that there exists a fixed point $x \in \partial \D$. We then will use to generate the arcs $U_i$ in \eqref{cylset} as before. 
To this end, pick an arbitrary point $p$ and consider $f^{-1}(p)$. This divides each $I_i$ into countably many arcs (as $f$ is essentially countable to one, and acts as a covering map on $I_i$ by the discussion surrounding Lemma~\ref{annul}) each of which is mapped by $f$ to $\partial \D \setminus \{p\}$. This implies that each subdivided arc of $I_i$ (which has end points in $f^{-1}(p)$) must contain a fixed point. 
Therefore, it is possible to find $x$.  As before, we use $x$ to define the sets $\{U_i\}$ (of which there will be countably many) that generate our partition.
\\

To prove recurrence results, we will need to understand the expansion rate of $f^n$ on the unit circle. This is determined by the derivative $f'$. As $f'$ is infinite on $\Sigma$ (see Lemma~\ref{svep}) we may not find bounds for $f'$ on an arbitrary arc, as it may contain points from $\Sigma$. However, as we will prove, we can obtain uniform bounds for $|(f^n)'|$ on each cylinder set $J_n$. Notice that as we will only consider the derivative away from the set of singularities, we can use the standard definition of derivative when considering $f'$.

We first will show that for all level one cylinder sets $J_{1_i} \in \mathcal{F}_1$, the map $f$ has bounded distortion.
\begin{lemma}\label{bdddist} Let $f$ be an infinite degree one-component inner function, and let $J_1 \in \mathcal{F}_1$ be a level-one cylinder set from the Markov partition. Then there exists a constant $C$, depending only on $f$, such that 
\begin{eqnarray}\label{bod1}
    \frac{1}{C} \: \leq \: \frac{|f'(a)|}{|f'(b)|}\:  \leq \: C \, 
\end{eqnarray}
for all $a,b \in J_1$.
\end{lemma}
\begin{remark}
    The constant $C$ depends on the distance of the set of singular values of $f$ to the unit circle.
\end{remark}

\begin{proof}
  To establish bounded distortion on $J_1$, we proceed in two steps. To begin, by Remark~\ref{altdefocf}, there exists $\xi > 0$ such that for any point $z$ on the unit circle and fixed pre-image $w$ of $z$, there exists a  holomorphic branch of $f^{-1}$ on $B(z,\xi) \subset {A}_{\rho}$ mapping $z$ to $w$. By compactness, we can therefore select finitely many points ${z_i}$ on $\partial \D$ (with corresponding preimages ${w_i}$) so that the collection of balls $$\mathscr{B}:=\{B(z_i,\xi/2)\}_{i=1}^N$$ covers $\partial \D$. The relevance of the choice $\xi/2$ will soon become apparent.
We obtain local distortion bounds on $f'$ by applying Koebe’s theorem to the preimages of the balls in $\mathscr{B}$. Second, we chain together these local estimates along the balls in $\mathscr{B}$ whose preimages cover $J_1$, yielding a uniform distortion bound across the whole of $J_1$.

In step one, $f^{-1}$ will denote the branch of the inverse sending $z_1$ to $w_1$.
Suppose now that $w_1:=f^{-1}(z_1) \in J_1$ lies in the preimage of $B(z_1,\xi/2)$. In order to use Koebe distortion on this ball, we start by applying a sequence of transformations to map $B(z_1,\xi) \to \D$. We apply the map $m_1: B(z_1,\xi) \to \C$ given by $m_1(z)=z-z_1$ to shift the ball to be centred at 0. We then rescale the radius to $1$ by composing with $m_2(z)=z/\xi$. The composition $f^{-1}\circ m_1^{-1}\circ m_2^{-1}: \D \to \C$ is univalent on the unit disk. We now normalise once more to make the origin a fixed point with derivative of modulus one.
On doing so, we obtain a map $g: \D \to \C$, given by
$$g(z) \: = \:  \frac{f^{-1}\circ m_1^{-1}\circ m_2^{-1}(z) \:- \: f^{-1}\circ m_1^{-1}\circ m_2^{-1}(0) }{|(f^{-1}\circ m_1^{-1}\circ m_2^{-1})'(0)|} \, $$
which is centred, univalent and $|g'(0)|=1$. Hence, $g$ satisfies all the conditions of the Koebe Distortion Theorem, and so we obtain the following bounds: 
\begin{eqnarray}\label{koeb}
    \frac{1-|z|}{(1+|z|)^3}  \: \leq \: |g'(z)|\:  \leq \: \frac{1+|z|}{(1-|z|)^3} \quad \text{ \ for all \ } z \in \D\, .
\end{eqnarray} 
For $\zeta \in B(z_1,\xi)$, let $z=(m_2 \circ m_1)(\zeta) \in \D$. Using the fact that $(f^{-1})'=\frac{1}{f'(f^{-1})}$, we obtain $$|g'(z)|= \frac{|f'(f^{-1}(z_1))|}{|f'(f^{-1}(\zeta))|} \, .$$
Since the bounds in \eqref{koeb} degenerate as $|z|\to 1$, we restrict attention to the smaller ball $B(0,1/2)$. In this case we have the explicit, finite bounds $4/27 \leq |g'(z)| \leq 12$ for all $z \in B(0,1/2)$. 
 Given any $x,y \in f^{-1}B(z_1,\xi/2)$, it then follows that 
\begin{eqnarray}\label{bod}
\frac{1}{81} \: \leq \: \frac{|f'(x)|}{|f'(y)|} \: \leq \: 81 \: \text{ \ for all \ }\: x,y \:  \in \: f^{-1}B(z_1,\xi/2) \, .
\end{eqnarray} 
We remark that inequality \eqref{bod} expresses the bounded distortion property of $f$ on the set $f^{-1}B(z_1,\xi/2)$, with absolute constant $81$.

We now wish to prove bounded distortion from points in a single arc $B(z_1,\xi/2)\cap \partial\D$ to any two points in a fixed $J_1 \in \mathcal{F}_1$. To do so, as outlined in the beginning of the proof, we choose a finite sequence of points $\{z_i\}_{1 \leq i \leq N}$ on $\partial\D$ and a branch of $f^{-1}$ such that each $z_i$ has a preimage $w_i \in J_1$, and the preimages of the balls $B(z_i,\xi/2)$ cover $J_1$. Then, for any $a,b \in J_1$, we can extract a chain of at most $N$  balls whose preimages join $a$ to $b$. Since none of the corresponding boundary points lie in $\Sigma$, we can apply Koebe Distortion Theorem on each ball. Iterating the bound \eqref{bod} along the chain yields
\begin{eqnarray*}
      \frac{1}{81^k} \leq \frac{|f'(a)|}{|f'(b)|} \leq 81^k \quad \text{ \ for some \ } 1 \leq k \leq N .
\end{eqnarray*}
As $a$ and $b$ were arbitrary, and $k \leq N$, we have for all $a,b \in J_1$
\begin{eqnarray*}
    \frac{1}{81^N} \: \leq \: \frac{|f'(a)|}{|f'(b)|}\:  \leq \: 81^N \, .
\end{eqnarray*}
Setting $C=81^N$, we are done.
\end{proof}

\begin{lemma}\label{kd} Let $f$ be an infinite degree one-component inner function, and let $J_1 \in \mathcal{F}_1$ be an interval from level one of the Markov partition of $\partial \D$ associated to $f$. Then, for the same $C$ as in Lemma~\ref{bdddist},
    $$
\frac{1}{C \lambda(J_1)} \leq |f'(w)| \leq \frac{C}{\lambda(J_1)} \, 
$$
for all $w \in J_1$.
\end{lemma}
\begin{proof}
By Lemma~\ref{bdddist}, it suffices to show that there exists $w \in J_1$ with $|f'(w)| \asymp 1/\lambda(J_1)$.
Write $f(e^{i\theta}) = e^{iS(\theta)}$, so $|f'(e^{i\theta})| = |S'(\theta)|$. Since $f$ is univalent on $J_1$ and $f(J_1) = \partial \D \setminus \{p\}$, the function $S: [0,2\pi) \to [0,2\pi)$ is monotone on $J_1$. Define $t := S/2\pi$, so that $t$ maps $\overline{J_1}$ bijectively onto $[0,1]$. Suppose $t$ sends $\overline{J_1}$ to $[x,y] \in [0,1]$, and note that by monotonicity $t(x)=0, t(y)=1$.
Then by the Mean Value Theorem,
\begin{eqnarray}\label{mvt1}
   1=|t(x)-t(y)| \: = \: |t'(c)||x-y| \, 
\end{eqnarray}
for some $c \in [x,y]$.  As $t'(\theta) = S'(\theta)/(2\pi) = f'(e^{i\theta})/(2\pi)$, we have $|t'(c)|=|f'(e^{ic})|$. Hence, 
\begin{eqnarray}\label{dis}
    |f'(e^{ic})|=1/\lambda(J_1) \, .
\end{eqnarray}
Finally, we can combine \eqref{bod1} and \eqref{dis} to obtain:
$$
\frac{1}{81^N \lambda(J_1)}\: \leq \: |f'(w)| \: \leq  \: \frac{ 81^N}{\lambda(J_1)} \, 
$$
for all $w \in J_1$.  As $J_1 \in \mathcal{F}_1$ was arbitrary, we are done.
\end{proof}
 The constants depend only on $N$ (and so only on $\xi$).
 We now show that under the extra hypothesis that $f$ is centred then  $|(f^n)'| \asymp 1/ \lambda(J_n)$ as in Lemma~\ref{kd} with the same constants.
\begin{lemma}\label{kdn}
Let $f$ be a centred, infinite degree one-component inner function. Then, for every $n \in \N$, every cylinder $J_n\in \mathcal{F}_n$ and every point $w \in J_n$, we have
   $$
\frac{1}{C \lambda(J_n)} \leq |(f^n)'(w)| \leq \frac{ C}{\lambda(J_n)} \, ,
$$
for the same constant $C$ as in Lemma~\ref{bdddist}.
\end{lemma}
\begin{proof}
We want to bound $|(f^n)'(w)|$.
By Lemma~\ref{annul}, as any branch of $f^{-1}$ is defined in a ball contained in $ A_{\rho}$, any branch of $f^{-n}$ is also well defined in this ball. Given a cylinder set at level $n$, say $J_n$, the point $z_1$ will have point $v_1 = f^{-1}(z_1)$ in $J_n$. We can apply the exact same argument with the map $g(z)$ instead involving $f^{-n}$ instead of $f^{-1}$. Hence, we have uniform bounds on $|(f^n)'(z)|$ for $z \in J_n$ with the same constants as for $J_1$.
\end{proof}
We finally prove that the dynamical system associated to  a centred, one-component inner function satisfies a conformality condition. In fact, the argument applies more generally to one-dimensional Ahlfors-regular systems with a full shift partition.

\begin{lemma}\label{lcondconf}
    Let $f$ is a centred, one-component inner function and $(\partial \D,\mathcal{B},\lambda,f,d)$ its associated dynamical system. There exists a constant $C$ such that for any $J_n \in \mathcal{F}_n$ and any $B(x_0,s) \subset J_n$
\begin{eqnarray}\label{condconf}
    B(f^n(x_0), C^{-1}K_{J_n}s) \subset f^nB(x_0,s) \subset B(f^n(x_0),CK_{J_n}s) 
\end{eqnarray}
    where $K_{J_n}:=\inf\limits_{z \in J_n}|(f^n)'(z)|$.
\end{lemma}
The proof uses the bounded distortion property proven in Lemma~\ref{bdddist} along with the Mean Value Theorem. The strategy is to show the end points  do  not get mapped by $f$ to the end points of the ball centred at $f(x_0)$. We do this by  bounding the distortion by $f$ of the distance of the end point of the original ball to the centre $x_0$.
\begin{proof}
We first consider when $n=1$.
As $B(x_0,s)$ is an open arc on $\partial \D$, it has end points $a,b$. On each cylinder set, $f(e^{i\theta})=e^{iS(\theta)}$ is monotone in $S(\theta)$. Therefore, as before we can use the Mean Value Theorem on $S$ and transfer it to a statement on $f$, and so there exists  $x_a \in  (a,x_0)$ and $x_b \in  (x_0,b)$ such that
\begin{eqnarray}
    d(f(a),f(x_0)) =|f'(x_a)|d(a,x_0)\quad \text{ \ and \ } \quad  d(f(b),f(x_0)) =|f'(x_b)|d(b,x_0) \, .
\end{eqnarray}
Further, note that
\begin{eqnarray}\label{eqn}
    d(f(a),f(x_0))+ d(f(x_0),f(b))=  d(f(a),f(b)) \geq K_{J_n}d(a,b) \, .
\end{eqnarray}
From the bounded distortion property (Lemma~\ref{bdddist}):
\begin{eqnarray}\label{rat}
    \frac{1}{81^N} \leq \frac{d(f(a),f(x_0))}{d(f(b),f(x_0))}=\frac{|f'(x_a)|}{|f'(x_b)|} \leq 81^N \, .
\end{eqnarray}
Combining \eqref{eqn} and \eqref{rat}:
$$
d(f(b),f(x_0)) \geq K_{J_1}d(b,a) - d(f(a),f(x_0)) \geq K_{J_1}d(b,a)-81^Nd(f(b),f(x_0)) \, .$$
We also have the same inequality for $d(f(a),f(x_0))$. Hence, as $d(a,b)=s$, we have
$$
B\Big(f(x_0),\frac{K_{J_1}s}{2(81^N+1)}\Big) \subset f(B(x_0,s))\, .
$$
This proves the lower inclusion. 

The upper inclusion follows similarly upon using the bounded distortion property once more to show that 
$$
 \frac{1}{81^N} \leq \frac{\sup_{z \in J_1}|f'(z)|}{\inf_{z \in J_1}|f'(z)|} \leq 81^N \, ,
$$
and so we find 
$$
f\big(B(x_0,s)\big) \subset B\left(f(x_0),\frac{81^{2N}K_{J_1}s}{2(81^N+1)}\right) \, .
$$
The same argument follows through for $n>1$, as the distortion constants are independent of the level of the partition.
\end{proof}
\begin{lemma}\label{lcondsum}
There exists an absolute constant $\kappa$ such that 
    \begin{eqnarray}\label{condsum}
        \sum_{J_n \in \mathcal{F}_n}K_{J_n}^{-1} < \kappa \quad \text{ \ for all \ } \,  n \in \N
    \end{eqnarray}
\end{lemma}
\begin{proof}
 As $f$ is uniformly expanding on $\partial \D$, we have, for all $z \in \partial \D$, $|f'(z)|>w>1$ and so $ |(f^n)'(z)|>w^n>1$. By the prior Koebe distortion estimates in Lemma~\ref{kdn}, $\lambda(J_n) \asymp 1/w^n$. As $$\sum_{J_n \in \mathcal{F}_n} \lambda(J_n) \leq 1 \, ,$$ we are done.
\end{proof}

\medskip

\subsection{Proof of Theorem~\ref{ABrecocif}}\label{proofrecocif} 
We will split the proof up into two cases based on the degree of $f$ being finite or infinite. In the finite case, the zero-one criterion follows directly from a theorem of He-Liao \cite[Theorem~1.4]{HeLi}, which applies to piecewise expanding maps.  In the one dimensional setting these are defined as follows:
\begin{definition}\label{pe} 
We say $t:[0,1] \to [0,1]$ is a piecewise expanding map if there exists a finite family $\{U_i\}_{i=1}^N$ of pairwise disjoint connected open subsets in $[0, 1]$ with $\bigcup_{i=1}^N \overline{U_i}=[0,1]$
such that the following statements
hold:
\begin{enumerate}
    \item  $t$ is
strictly monotonic and continuous on each $U_i$,
    \item There exists a constant $L \geq 1$ such that $\inf\limits_{x \in {U_i}}|t'(x)| \geq L$ for all $i$.
\end{enumerate}
\end{definition}
Note that this definition means the map $f$ has to have a finite Markov partition.
We state can now state the one dimensional version of \cite[Theorem~1.4]{HeLi}.
\begin{theorem}\label{hl}
Let $([0,1], \mathcal{B},\mu,t,d)$ be a measure-preserving, exponentially mixing dynamical system, with $\mu$ is an absolutely continuous with respect to Lebesgue measure $\lambda$. Further suppose that there exists an open set $V \subset [0,1]$ with $\mu(V) = 1$ such that the density $h$ of $\mu$, when restricted on $V$, is bounded from above by some $c \geq 1$ and bounded from below by $c^{-1}$.
 Then, 
\begin{eqnarray}\
   \mu(\cR(t,\{r_n\} ))= \begin{cases}
        0 &\text{ \ if}\ \  \sum_{n=1}^\infty r_n <\infty \, ,\\[2ex]
        1 &\text{ \ if}\ \  \sum_{n=1}^\infty r_n
        =\infty \, .
    \end{cases}
	\end{eqnarray}
\end{theorem}
We shall apply this directly to to prove a zero-one criterion for centred, finite Blaschke products. 

However, when $f$ has infinite degree, it also has an infinite Markov partition, and so does not satisfy Definition~\ref{pe}. Therefore, we can not apply Theorem~\ref{hl}. Instead, we directly prove the criterion.
Our strategy to prove a non-asymptotic zero-one criteria is the same for both the recurrence and shrinking target set. For the zero measure result, we apply Lemma~\ref{cbc} to the building block sets. For the full measure result, we first apply Lemma~\ref{dbc} to the building block sets, and then upgrade from positive to full measure result. To use the divergence result, Lemma~\ref{dbc}, we must show that the building block sets satisfy necessary quasi-independence on average condition \eqref{qioa}. 
In the recurrence case, verifying \eqref{qioa} is more subtle than in the shrinking target case: the sets $A_n$ cannot be expressed as the preimage of a ball, and so  quasi-independence does not follow directly from  $f$ being measure-preserving and uniform mixing. Here, the additional assumption of the system admitting a Markov partition with suitable regularity (see Section~\ref{Markov}) plays a crucial role. Moreover, our criterion depends on the summability of $r_n$ rather than $A_n$. Therefore, we require the partition to estimate $\lambda(A_n)$, and $\lambda(A_n \cap A_m)$ in terms of $r_n$ and $r_m$.
With the general strategies outlined for both cases, we begin the proof.
\begin{proof}[Proof of Theorem~\ref{ABrecocif}]
We consider two cases.\\[0.3em]
\textbf{$\text{Case 1:}$} $\,$ Consider the system $(\partial \mathbb{D}, \mathcal{B}, \lambda, b,d) $, where $b$ is a centred, degree $N < \infty$ Blaschke product. As shown in Section~\ref{Markov}, the system admits a finite Markov partition of the unit circle into a finite number of open arcs $\{U_i\}_{i=1}^N$, on which $b$ is injective, continuous, and uniformly expanding. Therefore, $b$ satisfies the conditions of Definition~\ref{pe}, modulo the domain being $\partial \mathbb{D}$ rather than $ [0,1]$.

Since $\lambda$ is normalised Lebesgue measure on $\partial \mathbb{D}$, the density condition in Theorem~\ref{hl} is trivially satisfied. Moreover, we know $b$ is exponentially mixing.
It remains to construct a smooth isomorphism  $\phi: \partial \mathbb{D} \to [0,1]$ by selecting the repelling fixed point used to generate the partition,  $x = e^{2\pi i \alpha} \in \partial \mathbb{D}$ and defining
$$
\phi(e^{2\pi i \theta}) := (\theta - \alpha) \bmod 1.
$$
This identifies $ \partial \mathbb{D}$ with $[0,1)$, sending $x \to 0$. Then the induced map $$t := \phi \circ b \circ \phi^{-1} : [0,1) \to [0,1)$$ is a piecewise expanding map (in the sense of Definition~\ref{pe}) conjugate to $b$, and preserving $\lambda$. The conjugacy also maps the Markov partition $\{U_i\}$ on $\partial \mathbb{D}$ to a finite partition $\{I_j\}$ of $[0,1)$ into open intervals, with $\overline{\bigcup I_j} = [0,1]$ and $ t|_{I_j}$ expanding and continuous. Therefore, all the hypotheses of Theorem~\ref{hl} are satisfied for $t$ , and hence for $b$, via the isomorphism. We thus obtain Theorem~\ref{ABrecocif} for finite degree maps as a corollary of Theorem~\ref{hl},
which completes the first case.\\[0.5em]
\textbf{$\text{Case 2:}$} $\,$
Now let $f$ be an infinite degree, centred one-component inner function. In order to show Condition~\eqref{qioa} is satisfied (and so prove the full measure law) we must estimate $\sum_{n=1}^\infty \lambda(A_n)$ and $\sum_{m,n=1}^\infty \lambda(A_n \cap A_m)$.
Note that in the proof of Theorem~\ref{ABconvin}, we already obtained estimates for $\lambda(A_n)$ and $\sum_{n=1}^\infty \lambda(A_n)$ in terms of $r_n$ -- see Lemma~\ref{an} and Lemma~\ref{suman}. The proof of Lemma~\ref{an} had two main steps. First, we used Lemma~\ref{trig} to locally estimate the recurrence building block sets via the preimages of balls. Then, we used exponential mixing.
 We will use the same strategy, that is, look locally at $A_n \cap A_m$ to obtain bounds to which we can apply the exponential mixing result; Proposition~\ref{expmixcomp}. 
However, if we apply Lemma~\ref{trig} to $A_m \cap B$, and use the subsequent bounds to estimate $A_n \cap A_m \cap B$, we will not be able to use exponential mixing because $ A_m \cap B$ is not necessarily an arc.
 Instead, we need to use the partition associated to $f$ to estimate $A_m \cap A_n$ locally. 
 
 To this end, we recall that $\mathcal{F}_m$ denotes the collection of cylinder sets of level $m$ of the Markov partition associated to $f$. Given a cylinder set $J_m \in \mathcal{F}_m$, we consider $A_m \cap J_m$.
  We now gather a lemma which will be key in establishing estimates on $\lambda(A_m \cap A_n)$.
    \begin{lemma}\label{boundball}
       Let $J_m$ be a level $m$ cylinder set. Then there exists a ball $B(x,r)$ such that  $$A_m \cap J_m \subset B(x,r) \cap J_m$$ for $r=r(m,J_m):= \frac{r_m}{K_{J_m}-1}$, and any $x \in A_m \cap J_m$.
    \end{lemma} 
    \begin{proof}
        Given a fixed $x \in  A_m \cap J_m$, observe that for any $ y \in A_m \cap J_m$, by the triangle inequality:
       \begin{equation}\label{est1}
        d(f^m(x),f^m(y)) \leq d(f^m(x),x)+ d(x,y)+d(f^m(y),y) \leq 2r_m + d(x,y) \, .
        \end{equation}
        As before, the Mean Value Theorem tells us that
        \begin{equation}\label{est2}
        d(f^m(x),f^m(y)) \geq K_{J_m}d(x,y) \, .
        \end{equation}
       Dividing \eqref{est1} by \eqref{est2} and rearranging gives
        $
    d(x,y) \leq 2r_m/(K_{J_m}-1) \, ,
        $
        and so $y \in B(x,r)$ as desired.
       
    \end{proof}
    With this lemma in hand, we can finally estimate $\lambda(A_m \cap A_n)$. Let $\kappa$ be as in Lemma~\ref{lcondsum}.
    \begin{proposition}\label{measint}
        For $m < n$, 
        $$
        \lambda(A_m \cap A_n) \,  \leq \,  \frac{4\kappa r_n}{\pi}\big(\frac{2r_m}{\pi}+Ke^{-(n-m)\tau}\big) + \frac{8\kappa r_m}{\pi}\big(\frac{2r_n}{\pi}+Ke^{-n\tau} \big)\, .
        $$
    \end{proposition}
    \begin{proof}[Proof of Proposition~\ref{measint}]
    By Lemma~\ref{boundball}, we have the containment 
        $$
        A_m = \bigcup\limits_{J_m \in \mathcal{F}_m} A_m \cap J_m \subset \bigcup\limits_{J_m \in \mathcal{F}_m} J_m \cap B(x,r) \, 
        $$
        for $r= \frac{r_m}{K_{J_m}-1}$.
If $r< r_n$, then we can apply Lemma~\ref{trig} with the ball $B(z,r_1)=J_m \cap B(x,r)$ and obtain 
\begin{eqnarray}\label{eq1}
    A_n \cap J_m \cap B(x,r) \subset f^{-n}B(x,2r_n)  \cap J_m \cap B(x,r) \, , 
\end{eqnarray}
where we have bounded $r_1+r \leq 2r_n$ as $r_1\leq r< r_n$.
Although tempting, we should not apply Proposition~\ref{expmixcomp} yet to the RHS of \eqref{eq1}. If we do, then our estimate will involve a term with no dependency on $J_m$, and so when we sum over all  $J_m \in \mathcal{F}_m$ (of which there are infinite), we cannot obtain a good enough bound. 
Instead, observe that given any ball $B(x_0,s) \subset J_m$, Lemma~\ref{lcondconf} gives
$$
\lambda(f^mB(x_0,s)) \asymp K_{J_m}\lambda(B(x_0,s)) \, .
$$
As any open set can be written as a countable disjoint union of balls, this holds for any open subset of $J_m$. In particular, $J_m \cap B(x,r)$ is an open subset of $ J_m$, so we get the following bound on the measure of the set in the RHS of \eqref{eq1} :

\begin{align}\label{eq2}
   \lambda\left(f^{-n}B(x,2r_n) \cap J_m \cap B(x,r)\right) \: & \lesssim & \frac{1}{K_{J_m}}  \lambda\left(f^{m-n}B(x,2r_n)  \cap f^m(J_m \cap B(x,r))\right) \, .  
\end{align}  
Once more, by Lemma~\ref{lcondconf} we have the containment $f^m (J_m \cap B(x,r)) \subset B(f^m(x),r_m)$, and so 
\begin{eqnarray*}
    \lambda\left(f^{m-n}B(x,2r_n)  \cap f^m\big(J_m \cap B(x,r)\big)\right) &\leq& \lambda\big(f^{m-n}B(x,2r_n)  \cap B(f^m(x),r_m)\big) 
    \\[2ex]& \leq & \lambda(B(x,2r_n))\left(\lambda\big(B(f^m(x),r_m)\big)+ Ke^{-(n-m)\tau} \right)
    \\[2ex] & \leq & \frac{4r_n}{\pi} \big(\frac{2r_m}{\pi} + Ke^{-(n-m)\tau} \big)
\end{eqnarray*}
where we have used Proposition~\ref{expmixcomp} for the second inequality. Combining this with \eqref{eq2} gives 
$$
\lambda(A_n \cap J_m \cap B(x,r)) \leq  \frac{4r_n}{K_{J_m}\pi} \big(\frac{2r_m}{\pi} + Ke^{-(n-m)\tau} \big) \, .
$$
Finally, we sum over all $J_m$ to obtain 
\begin{align}\label{sum0}
    \sum\limits_{\substack{J_m \in \mathcal{F}_m  \\ r \leq r_n}} \lambda(A_n \cap J_m \cap B(x,r)) \leq \sum\limits_{\substack{J_m \in \mathcal{F}_m  \\ r \leq r_n}} \frac{4r_n}{K_{J_m}\pi} \big(\frac{2r_m}{\pi} + Ke^{-(n-m)\tau} \big) \, .
\end{align}
 We use Lemma~\ref{lcondsum}, and conclude that
\begin{align}\label{sum1}
    \sum\limits_{\substack{J_m \in \mathcal{F}_m  \\ r \leq r_n}} \lambda(A_n \cap J_m \cap B(x,r)) \leq \frac{4\kappa r_n}{\pi}\big(\frac{2r_m}{\pi}+Ke^{-(n-m)\tau}\big) \, .
\end{align}
This concludes the case where $r \leq r_n$. 
\\
When $r > r_n$, we cover of $B(x,r)$ by balls of radius $r_n$, say $\{B(x_i,r_n)\}_{n=1}^{q}$ where $q:=q(n,m, J_m)=\lceil \frac{r}{r_n} \rceil $. Because $r_n<r \approx r_m/\lambda(J_m)$, we can be less careful with our local estimates. With this in mind, we use Lemma~\ref{trig} and exponential mixing on each ball: 
\begin{align*}
    \lambda\big(A_n \cap B(x_i,r_n) \big) & \leq  \lambda\big(B(x_i,r_n) \cap f^{-n}B(x_i,2r_n)\big) \\[1ex] & \leq \lambda(B(x_i,2r_n)) \big(\lambda(B(x_i,r_n))+ Ke^{-n\tau} \big) \\[1ex] & \leq \frac{4r_n}{\pi}\big(\frac{2r_n}{\pi}+ Ke^{-n\tau} \big) \, .
\end{align*}
Therefore, 
$$
\lambda(A_n \cap J_m \cap B(x,r)) \leq \sum\limits_{n=1}^{q}\lambda(A_n \cap B(x_i,r_n)) \leq \frac{4qr_n}{\pi}\big(\frac{2r_n}{\pi}+ Ke^{-n\tau} \big) \, .
$$

Finally, summing over all level $m$ cylinder sets, and using the bound $q \leq \frac{2r_m}{r_n K_{J_m}}$:
\begin{align}\label{sum2}
    \lambda(A_n \cap J_m \cap B(x,r)) &\leq \sum\limits_{\substack{J_m \in \mathcal{F}_m  \\ r > r_n}} \lambda(A_n \cap J_m \cap B(x,r)) \nonumber 
    \\[2ex] & \leq \frac{8\kappa r_m}{\pi}\big(\frac{2r_n}{\pi}+Ke^{-n\tau} \big) \, .
\end{align}
Combining \eqref{sum1} and \eqref{sum2} gives 
$$
\sum\limits_{J_m \in \mathcal{F}_m } \lambda(A_n \cap J_m \cap B(x,r)) \leq \frac{4\kappa r_n}{\pi}\big(\frac{2r_m}{\pi}+Ke^{-(n-m)\tau}\big) + \frac{8\kappa r_m}{\pi}\big(\frac{2r_n}{\pi}+Ke^{-n\tau} \big)\, .
$$
    \end{proof}
We are now in a position to show that the recurrence set has positive measure.

\begin{lemma}\label{posmeas} If 
$
\sum\limits_{n=1}^\infty r_n = \infty
$, then
    $\lambda(\mathcal{R}(f,\{r_n\}))>0$.
\end{lemma}

    \begin{proof}
    First, note that by the Divergence Borel-Cantelli Lemma, if we prove both that the $A_n$s satisfied the quasi-independence on average condition, and that $\sum_{n=1}^\infty \lambda(A_n)=\infty$, then we can conclude $\lambda(\cR(f,\{r_n\}))=\lambda(\limsup_{n \to \infty} A_n)>0$. The latter assumption follows from the fact that $\sum_{n=1}^\infty r_n=\infty$ with the forward direction of Proposition~\ref{maincon1iff}. It remains to show that condition \eqref{qioa} is satisfied,  i.e. that there exists a constant $C>1$ such that for infinitely many $N$
    $$
    \sum_{m,n=1}^N\lambda(A_n \cap A_m) \leq C\Big(\sum_{n=1}^N \lambda(A_n)\Big)^2 \, .
    $$
 Lemma~\ref{suman} and the assumption that the sum of $r_n$ diverges gives, for $N$ large enough,
$$
\Big(\sum_{n=1}^N \lambda(A_n)\Big)^2 \: \geq \: \frac{1}{4\pi^2}\Big(\sum_{n=1}^N r_n\Big)^2-\frac{\gamma}{\pi}\sum_{n=1}^N r_n+\gamma^2 \: \geq \: \frac{1}{8\pi^2}\Big(\sum_{n=1}^N r_n\Big)^2 \, .
$$
We turn to estimating the other summation. We use Proposition~\ref{measint} to get
\begin{align*}
     2\sum_{1 \leq m<n\leq N}\lambda(A_n \cap A_m) \, &\leq \, \frac{48\kappa}{\pi}\sum_{1 \leq m<n \leq N}r_n r_m + \frac{8\kappa K}{\pi} \sum_{1 \leq m<n \leq N} r_ne^{-(n-m)\tau} \\[2ex]& \hskip30pt + \frac{16\kappa K}{\pi}\sum_{1 \leq m<n \leq N}r_me^{-n \tau} \, .
\end{align*}
Now, we  can use the partial geometric series formula to estimate the latter two terms, and obtain
\begin{align*}
     2\sum_{1 \leq m<n\leq N}\lambda(A_n \cap A_m) \, &\leq \, \frac{48\kappa}{\pi}\sum_{1 \leq m<n \leq N}r_n r_m + \frac{8\kappa K}{\pi(1-e^{-\tau})} \sum_{n=1}^N r_n  + \frac{16\kappa K}{\pi(1-e^{-\tau})}\sum_{m=1}^Nr_m\, .
\end{align*}
 Therefore, for both this and  $\left( \sum_{n=1}^N \lambda(A_n) \right)^2 $, we have that for all sufficiently large $N$, the leading term is $ (\sum_{n=1}^\infty r_n)^2 $. Hence, we can find a constant $C$ to satisfy the hypothesis of Lemma~\ref{dbc}, and conclude that 
$$
\lambda(\limsup\limits_{n \to \infty} A_n)= \lambda(\mathcal{R}(f,\{r_n\}))>0 \, .
$$ 
\end{proof}
With Lemma~\ref{posmeas} in hand, it remains to show that positive measure for $\mathcal{R}(f,\{r_n\})$ implies full measure.  Let $L=\mathcal{R}(f,\{r_n\})$ and $L_m= \{z \in \partial \D: f^n(z) \in B(z,r_{n+m}) \text{ \ for i.m \ } n \in \N\}$.
Notice that $f^{-m}(L_m)= \{z \in \partial \D: f^{n+m}(z) \in B(z,r_{n+m}) \text{ \ for i.m \ } n \in \N\} = L$. As $n$ was arbitrary, by Proposition~\ref{contractpom}, we conclude that the recurrence set has full measure. 
\end{proof}

\medskip

\section{Autonomous vs non-autonomous}\label{autnonaut}
We finish by briefly highlighting the differences between autonomous and non-autonomous results for the shrinking target and recurrence set. As previously mentioned, for an autonomous system with an ergodic map, and $r_n=c$ for all $n \in \N$, we are guaranteed full measure of the shrinking target and recurrence set. This followed from Poincar\'{e} Recurrence, and the Ergodic Theorem, neither of which carry directly over to the non-autonomous setting.
Indeed, given a non-autonomous system with a sequence of maps $T_\N$,  assuming each map is ergodic is not enough to guarantee a zero-one criterion for either set of interest. We will use the following theorem to construct a specific example of this phenomena:
\begin{theorem}[\cite{gust}, Theorem~1.3]\label{thmfer}
     Let $b_n$ be a sequence of centred Blaschke products with degree bounded by some $D$, and $b'_n(0)>0$. Let $B_n=b_n \circ\dots\circ b_1$. If 
    $$
    \sum\limits_{n=1}^ \infty (1-b'_n(0))\log\big(\frac{1}{1-b'_n(0)}\big) < \infty \, ,
    $$
   then we have $B_n \to b$ in $L^1$, where $b$ is a Blaschke product.
\end{theorem}
    Notice that if $B_\N$ satisfies the hypothesis of Proposition~\ref{expmixcomp} then the sum in Theorem~\ref{thmfer} diverges. In particular, in the autonomous case where $b_n$ is the same map for all $n$, the sum diverges.
\begin{example}
    Let $$b_n(z)=z\frac{(1-1/n^2)-z}{1-(1-1/n^2)z}$$ for $n \geq 2$ be a sequence of centred Blaschke products.  Note that the zeros of each $b_n$ are at  $0$ and $1-1/n^2$. Each $b_n$ is centred (so preserves $\lambda$) and not a rotation. Hence, they are each ergodic -- even exact \cite{DM}. Let $B_n=b_n \circ\dots\circ b_2$. By Theorem~\ref{thmfer}, as 
    $$
     \sum\limits_{n=1}^ \infty  (1-b'_n(0))\log(\frac{1}{1-b'_n(0)}) =  \sum\limits_{n=1}^ \infty  1/n^2 \log(\frac{1}{n^2}) 
     < \infty
    $$
    we know $B_n \to b$ in $L^1$ on $\partial \D$, and so there exists a subsequence $B_{n_j} \to b$ pointwise. Hence, without loss of generality, we assume the original $B_n \to b$ pointwise. Consider the associated shrinking target when $r_n=c$ for all $n \in \N$.
    As $B_n \to b$, we can rewrite this set as  
    $$\mathcal{W}(B_\N,x_0,c)= \{z \in \partial \D : b(z) \in B(x_0,c)\}=b^{-1}B(x_0,c) \, .$$
    However, $\lambda(b^{-1}B(x_0,c))=\lambda(B(x_0,c))=\frac{c}{\pi}\neq 0,1$ unless $c=0,\pi$, and so we do not have a full measure result. Similarly, for the recurrence set 
$$\lambda(\mathcal{R}(B_\N,c)) = \lambda\big(\{z: b(z) \in B(z,c)\}\big) \, .$$ Unless $b=I$, the identity, the recurrence set will only have full measure for certain values of $c$. It is easier to see this by considering an example where the maps converge to a rotation, say $b(z)=ze^{i\beta}$. If we choose $c< \beta$, then the recurrence set has zero measure. 

This discussion illustrates that, unlike in the autonomous case, imposing the conditions $r_n=c$ for all $n$ and that the sequence of maps each being ergodic, is not sufficient to conclude full measure of the shrinking target or the recurrence set. 
 Moreover, if we choose $r_n$ such that  $r_n \to 0$, then if $b \neq I$, we have $\lambda(\mathcal{W}(B_\N,\{x_n, r_n\}))=\lambda(\cR(B_\N,\{r_n\}))=0$ regardless of whether $\sum_{n=1}^\infty r_n$ converges.
\end{example}
\section*{Acknowledgements}
\thispagestyle{empty}
I would like to thank Malabika Pramanik for her support during my undergraduate research project; Rod Halburd, who supervised my master’s project and introduced me to the beautiful world of complex dynamics; and my PhD supervisor, Holly Krieger, who signed herself up to four years of me! Thanks also to Anna Jové and Oleg Ivrii for their helpful conversations over the past year.

This work was supported by the Additional Funding Programme for Mathematical Sciences, delivered by EPSRC (EP/V521917/1) and the Heilbronn Institute for Mathematical Research.

\clearpage

\newpage
\bibliographystyle{abbrv}

\end{document}